\documentclass[12pt]{amsart}
\usepackage[all,cmtip]{xy}
\usepackage{pb-diagram,pb-xy}
\usepackage{amsmath,amscd,amsthm,amsfonts, amssymb,amsxtra}
\usepackage{graphicx}
\usepackage{hyperref}
  \hypersetup{colorlinks=true,citecolor=blue}
 \usepackage[font=scriptsize]{caption}

\CDat

\SelectTips{cm}{}
\newtheorem{thm}{Theorem}[section]  
\newtheorem*{un-no-thm}{Theorem}
\newtheorem{cor}[thm]{Corollary}     
\newtheorem{lem}[thm]{Lemma}         
\newtheorem{prop}[thm]{Proposition} 
 
\newtheorem{conj}[thm]{Conjecture}

\newtheorem{bigthm}{Theorem}

\newtheorem{bigcor}[bigthm]{Corollary}

\theoremstyle{definition}
\newtheorem{defn}[thm]{Definition}   

\theoremstyle{definition}

\theoremstyle{definition}

\theoremstyle{remark}
\newtheorem{rem}[thm]{Remark}        

\newtheorem*{acks}{Acknowledgements}

\newtheorem{ex}[thm]{Example}

\DeclareMathOperator{\End}{end}
\DeclareMathOperator{\Top}{top}
\DeclareMathOperator{\phys}{phy}

\begin{document}
\title[Quantization of fluctuating currents]{Algebraic topology and 
the quantization of fluctuating currents}
\date{\today \\ \indent \footnotesize 2010 {\it Mathematics Subject Classification.} Primary: 82C31, 60J28, 55R80,  Secondary: 05C21, 55R40, 82C41.}

\author {Vladimir Y. Chernyak}
\address{Department of Chemistry, Wayne State University, Detroit, MI 48202}
\email{chernyak@chem.wayne.edu} 
\author{John R. Klein}
\address{Department of Mathematics, Wayne State University, Detroit, MI 48202}
\email{klein@math.wayne.edu}

\author{Nikolai A. Sinitsyn}
\address{Theoretical Division, Los Alamos National Laboratory, Los Alamos, NM 87545,
and New Mexico Consortium, Los Alamos, NM 87545}
\email{sinitsyn@lanl.gov}

\begin{abstract} We give a new approach to
 the study of statistical mechanical systems:  
algebraic topology is used to investigate
the statistical distributions of stochastic currents generated in graphs.
In the adiabatic and low temperature limits we will demonstrate that quantization
of current generation occurs.
\end{abstract}

\maketitle
\setlength{\parindent}{15pt}
\setlength{\parskip}{1pt plus 0pt minus 1pt}
\def\bdot{\bold .}
\def\Sp{\bold S\bold p}
\def\vo{\varOmega}
\def\smsh{\wedge}
\def\^{\wedge}
\def\flush{\flushpar}
\def\id{\text{\rm id}}
\def\dbslash{/\!\! /}
\def\codim{\text{\rm codim\,}}
\def\:{\colon}
\def\holim{\text{holim\,}}
\def\hocolim{\text{hocolim\,}}
\def\cal{\mathcal}
\def\Bbb{\mathbb}
\def\bold{\mathbf}
\def\simtwohead{\,\, \hbox{\raise1pt\hbox{$^\sim$} \kern-13pt $\twoheadrightarrow \, $}}
\def\codim{\text{\rm codim\,}}
\def\stableto{\mapstochar \!\!\to}
\let\Sec=\S
\def\Z{\mathbb Z}









\setcounter{tocdepth}{1}
\tableofcontents

\section{Introduction}
\label{sec:intro}
In statistical physics and chemistry, especially in the study of classical stochastic systems at the intermediate length scale, a {\it master
equation} governs the time evolution of states, in which transitions between states
are treated probabilistically.
 In its most compact form, the master equation 
 is $\dot{\mathbf p} = \tau_D H_\beta{\mathbf p}$, where ${\mathbf p(t)}$
 is a one parameter family of probability
 distributions on the state space, $\tau_D$ is a constant that
  represents total driving time and 
 $H_\beta$ is the master operator,
which depends both on time $t$ and a number $\beta$ representing inverse temperature.

We will be interested in varying the parameters $\tau_D$ and $\beta$.
When $\tau_D$ is made large,  the duration of time it takes to traverse the driving path
is large, and one refers to this process as {\it adiabatic} (or slow) driving.
The limiting case $\tau_D \to \infty$ is called  the {\it adiabatic driving limit}.  Similarly, one can consider
the effects of low temperature on the system;
the limiting case $\beta\to \infty$ is referred to as  the {\it low temperature limit}. 

Associated with the formal solution
of the master equation is an {\it average current vector} which represents the 
probability flux of a given initial distribution of states. In our first physics paper \cite{CKS1}, we argued that for generic periodic
driving protocols,
taking first  the adiabatic limit and subsequently the low temperature limit 
results in an average current vector  having {\it integer} components. 
This quantization phenomenon has been observed in a variety of applications, including 
electronic turnstiles, ratchets, molecular motors and heat pumps (cf.\ the 
bibliography of
\cite{CKS1}).
One of the
purposes of the current paper is to give this result a 
mathematically rigorous foundation. Our second aim 
is to explain how algebraic topology enters the picture in
an essential way.
\medskip

We now develop a mathematical formulation 
of our main results.
Consider a particle taking a continuous time random walk on a connected finite graph
$\Gamma$.
The particle starts at a vertex $i$, say, and at a random waiting time it jumps
to an adjacent vertex $j$ where it waits again and so forth. 
Aside from the choice of inverse temperature $\beta$, such a process is determined by choosing a collection 
of real parameters, one assigned to each vertex (well energies) and to each edge (barrier energies) of the graph. 
The space of these
parameters is denoted by $\cal M_{\Gamma}$; it has the structure of a real vector space
whose dimension $d$ is the number of vertices plus the number of edges of $\Gamma$.

Current generation occurs when the parameters are allowed to vary 
in a one parameter family.\footnote{This fits with the modeling of
physical and chemical processes: artificial machines at the mesoscopic scale depend on external parameters such as electric fields, temperature, pressure and chemical potentials
which typically vary in time.} 
 We consider such a family to be parametrized by an interval $[0,\tau_D]$, in which the number $\tau_D$ represents total driving time.  
If the value of the parameters at the endpoints coincide, one obtains a periodic driving protocol; it can be represented as a pair $(\tau_D,\gamma)$ 
in which $\gamma\: [0,1] \to M_{\Gamma}$ is a smooth loop (equivalently,
it is a smooth Moore loop).
For each periodic driving protocol $(\tau_D,\gamma)$
and each $\beta$ we can associate a class $Q_{\tau_D,\beta}(\gamma) \in H_1(\Gamma;\Bbb R)$ lying in the first homology of the graph
with real coefficients. The class is defined in terms of the formal solution
of the master
equation and is called the {\it average current}
generated by the triple $(\tau_D,\gamma,\beta)$. 
Physically, the average current 
is a measurement of the ``pumping'' by external forces acting on the system.

The assignment
$\gamma\mapsto Q_{\tau_D,\beta}(\gamma)$ describes a smooth map
\[
 Q_{\tau_D,\beta}\: L\cal M_{\Gamma} \to H_1(\Gamma;\Bbb R)\, ,
\]
where $L{\cal M}_{\Gamma}$ is the space of smooth unbased loops in $\cal M_{\Gamma}$ with the
Whitney $C^\infty$ topology. By taking the adiabatic limit $\tau_D \to \infty$, and using the Adiabatic Theorem (Corollary \ref{cor:adiabatic_theorem}), we obtain
a smooth map 
\[
Q_\beta\: L\cal M_{\Gamma} \to H_1(\Gamma;\Bbb R)
\]
which does not depend on the parameter $\tau_D$. 
We call the latter the {\it analytic current} map.

If we subsequently take the low temperature limit $\beta\to \infty$, it turns out that the resulting map is not everywhere defined. 

\begin{defn} A loop $\gamma\in L{\cal M}_{\Gamma}$ is said to be {\it intrinsically robust}
if there is an open neighborhood $U$ of $\gamma$ such that the low temperature limit
\[
Q:= \lim_{\beta\to \infty} Q_\beta
\] is well-defined and constant on $U$.
The subspace of $L{\cal M}_{\Gamma}$  consisting of the intrinsically robust loops is denoted
by $\check{L}{\cal M}_{\Gamma}$. \end{defn}

The main result of this paper is a quantization result for $Q$.

\begin{bigthm}[Pumping Quantization Theorem] \label{thm:strongPQT}
The image of the map
\[
Q \: \check L{\cal M}_{\Gamma}
\to H_1(\Gamma;\Bbb R)
\]
is contained in the integral lattice $H_1(\Gamma;\Bbb Z)\subset H_1(\Gamma;\Bbb R)$.
 \end{bigthm}

A version of this statement was observed earlier in our statistical mechanics papers \cite{CKS1}, \cite{CKS2}, and we will provide a rigorous proof below. A companion to the Pumping Quantization Theorem is the Representability Theorem, which gives a characterization of the
space of intrinsically robust loops:

\begin{bigthm}[Representability Theorem]\label{thm:representability}  There 
is a topological subspace 
$\check{\cal D} \subset {\cal M}_{\Gamma}$ such that
\[
\check L{\cal M}_{\Gamma} = L({\cal M}_{\Gamma}\setminus \check{\cal D})\, .
\]
Consequently, the space of intrinsically robust loops is a loop space.
\end{bigthm}

The subspace $\check{\cal D}$ is called the {\it discriminant},
and its complement $\check {\cal M}_{\Gamma} := {\cal M}_{\Gamma}\setminus \check{\cal D}$
is called the space of {\it robust parameters}. 

\begin{bigthm}[Discriminant Theorem] \label{thm:discriminant} The one point compactification
of the discriminant, i.e., $\check{\cal D}^+$, has the structure of a finite regular CW complex of dimension $\dim {\cal M}_{\Gamma} -2 =$
$d-2$. In particular, the inclusion $\check {\cal M}_{\Gamma}
\subset {\cal M}_{\Gamma}$ is open and dense.
\end{bigthm}

\begin{rem} A CW complex is said to be regular if its characteristic maps
are embeddings. By \cite[p.\ 534]{Hatcher}, such spaces have the structure of polyhedra.
In particular, $\check{\cal D}^+$  is  a finite polyhedron of dimension $d-2$.
We will explicitly describe the characteristic maps of $\check{\cal D}^+$ in \S\ref{sec:discriminant_theorem}.
\end{rem}

\begin{bigcor} \label{cor:generic} The inclusion $\check{L}{\cal M}_{\Gamma} \subset L{\cal M}_{\Gamma}$
is generic. In particular, a smooth loop $\gamma\in  L{\cal M}_{\Gamma}$ 
can always be infinitesimally perturbed to an intrinsically robust smooth 
loop $\gamma_1 \in \check{L}{\cal M}_{\Gamma}$.
\end{bigcor}

Another main result of this paper is to give an algebraic topological model for 
the map $Q$:

\begin{bigthm}[Realization Theorem] \label{thm:realization} There is a weak map
\[
\check q\: \check {\cal M}_{\Gamma} \to |\Gamma|\,
\]
such that the composite
\[
L\check {\cal M}_{\Gamma}  @>>>
H_1(\check {\cal M}_{\Gamma};\Bbb Z) @> {\check q}_* >>
H_1(\Gamma;\Bbb Z)
\]
coincides with $Q$, where $L\check {\cal M}_{\Gamma}  @>>>
H_1(\check {\cal M}_{\Gamma};\Bbb Z)$ is the map that sends a free loop
to its homology class.
\end{bigthm}

\noindent (Here, $|\Gamma|$ is the geometric realization of $\Gamma$. Recall that a {\it weak map}
$X \to Y$ is a diagram $X \leftarrow X' \to Y$, in which $X'\to X$ is a weak homotopy equivalence.)

\begin{rem} As long as $\Gamma$ has a non-trivial cycle, the homomorphism
 ${\check q}_*\: H_1(\check {\cal M}_{\Gamma};\Bbb Z) \to 
H_1(\Gamma;\Bbb Z)$ is non-trivial (cf.\ Remark \ref{rem:non-triviality}).
In particular, the map $Q$ is non-trivial.
\end{rem}

\begin{rem} Observe that
Theorem \ref{thm:realization} implies Theorem \ref{thm:strongPQT}. 
However, our actual procedure is to verify
Theorem \ref{thm:strongPQT} first and thereafter use
the tools of that proof  to establish
 Theorem \ref{thm:realization}. 
\end{rem}

Our final result gives an interpretation of the homomorphism ${\check q}_*$ in terms of the first Chern class of a certain line bundle.  Its formulation requires some preparation. A {\it weak complex line bundle} over a space $X$ is a pair $(\xi,h)$ consisting of a 
weak homotopy equivalence $h\:X' @> {\sim} >> X$ and a complex line bundle $\xi$ over $X'$
(in terms of classifying spaces, this is the same thing as 
specifying a weak map  $X \to BU(1)$). When the weak equivalence $h$ is understood,
we sometimes drop it from the notation and simply refer to $\xi$ as a weak complex line bundle
over $X$. Since $h$ is a cohomology isomorphism, there is no loss in considering
the first Chern class of $\xi$ as lying in $H^{2}(X;\Bbb Z)$.

Now suppose that $X = Y \times Z$. Then slant product with $c_1(\xi)$
defines a homomorphism
$
c_1(\xi)/\: H_1(Z;\Bbb Z) \to H^1(Y;\Bbb Z)\, .
$
Let $S(\Gamma) = U(1)^n$ be the $n$-torus, where $n$ is the first Betti number of $\Gamma$.  


\begin{bigthm}[Chern Class Description]\label{thm:Chern_description} 
There exists a weak complex line bundle $\xi$ 
on the
cartesian product $S(\Gamma) \times \check {\cal M}_{\Gamma}$
such that
\[
H_1(\check{\cal M}_{\Gamma};\Bbb Z) @> c_1(\xi)/ >> 
 H^1(S(\Gamma);\Bbb Z) = H_1(\Gamma;\Bbb Z)
\]
coincides with ${\check q}_*$.
\end{bigthm}


\begin{acks}
We are grateful to Misha Chertkov and Mike Catanzaro for useful discussions and comments.
The first and second authors wish to acknowledge the Center for Nonlinear Studies as well as
the New Mexico Consortium for partially supporting this research. Work at the New Mexico Consortium  was funded by NSF grant NSF/ECCS-0925618. 
This material is based upon work supported by the National Science Foundation 
under Grant Nos.\ CHE-1111350, DMS-0803363 and DMS-1104355.
\end{acks}

\section{Preliminaries}
\label{sec:preliminary}

\subsection*{Graphs}
We fix a connected finite graph
\[
\Gamma = (\Gamma_0,\Gamma_1) \, ,
\] 
where $\Gamma_0$ is the set of vertices and $\Gamma_1$ is
the set of edges. Here we are allowing multiple edges
between vertices and also edges linking a vertex to itself (loop edges).  
The entire structure of  $\Gamma$ is then given by specifying a function
\[
d\: \Gamma_1 \to \Gamma_0^{(2)}
\] 
which assigns to an edge the set of vertices which it connects ($\Gamma_0^{(2)}$
denotes the two-fold symmetric product of the set of vertices). 
For convenience, 
we fix a total ordering for $\Gamma_0$. Then $d$ lifts to 
a map $(d_0,d_1)\: \Gamma_1\to \Gamma_0 \times \Gamma_0$ 
in the sense that $d(e) = \{d_0(e),d_1(e)\}$, with
$d_0(e) \le d_1(e)$, where $d_0(e) = d_1(e)$ if and only if $e$ is a loop edge.
The maps $d_i\: \Gamma_1 \to \Gamma_0$, for $i = 0,1$ are called face operators.

The {\it geometric realization} of $\Gamma$ is the one dimensional
CW complex $|\Gamma|$ given by the amalgamated union
\[
\Gamma_0 \cup (\Gamma_1 \times [0,1])
\]
in which we identify $(e,i) \in \Gamma_1 \times \{0,1\}$
with $d_i(e) \in \Gamma_0$ for $i =0,1$.

\subsection*{Populations and currents}
\label{subsec:Markov-chain}

\begin{defn}
The space of {\it population vectors} $C_0(\Gamma;\Bbb R)$ is the real vector space  with basis 
$\Gamma_0$ and the space of {\it current vectors} $C_1(\Gamma;\Bbb R)$ is the real vector space with basis 
$\Gamma_1$.  If $\bold p$ is a population vector and
$i\in \Gamma_0$, then $\bold p_i$ denotes
the $i$-th component of $\bold p$. Likewise, if $\bold J$ is a
current vector $\alpha\in \Gamma_1$ then $
\bold J_\alpha$ denotes the $\alpha$-th component of $\bold J$. 
\end{defn}

The boundary operator
\[
\partial \: C_1(\Gamma;\Bbb R) \to C_0(\Gamma;\Bbb R)
\]
is given on basis elements by $\partial(\alpha) = d_0(\alpha) - d_1(\alpha)$.
Then $C_*(\Gamma;\Bbb R)$ is the cellular chain complex of $\Gamma$ over the vector
space of real numbers. The spaces  $C_i(\Gamma;\Bbb R)$ are smooth
manifolds and $\partial$ is a smooth map which is  a cellular chain analog
 of the divergence operator.
If a current ${\mathbf J}$ lies in $H_{1}(\Gamma;\Bbb R) := \text{ker}(\partial)$, 
we say that it is {\it conserved.}

The subspace $\bar{C}_{0}(\Gamma;\Bbb R) \subset C_0(\Gamma;\Bbb R)$ consisting of
population vectors ${\mathbf p}$ such that
$\sum_{i \in \Gamma_0} \bold p_i = 1$ is called the space of {\it normalized} population vectors; these can be viewed as discrete probability density functions on the space of states
$\Gamma_0$.
The subspace $\tilde{C}_{0}(\Gamma;\Bbb R) \subset C_0(\Gamma;\Bbb R)$
of those $\mathbf p$ such that $\sum_i {\mathbf p}_i = 0$ is called the space of 
{\it zero population} vectors.

\section{Driving Protocols}
\label{sec:driving_protocols}
Stochastic processes for periodic driving are governed by the 
``master equation'' which is a certain linear first order differential equation
acting on time dependent families of population vectors (see \cite[Ch.\ V]{vanKampen}). 
Our master equation is a combinatorial analog of the Fokker-Planck equation in Langevin dynamics \cite[Chap.\ VIII]{vanKampen}.

\subsection*{The space of parameters}
\label{subsec:driving-loop-spaces}
The {\it space of parameters} for $\Gamma$ is  
the real vector space 
\[
\cal M_{\Gamma}
\] 
consisting of ordered pairs $(E,W)$ where
$E\: \Gamma_0 \to \Bbb R$ and $W\: \Gamma_1\to \Bbb R$ are real-valued functions. 
The function $E$ is known as the set of
{\it well energies} and $W$ is is known as the set of 
{\it barrier energies}. We sometimes write
 $(E_i,W_\alpha)$ 
 for the value of $(E,W)$ at $(i,\alpha)\in \Gamma_0\times \Gamma_1$.

\begin{rem}
Notice that $\cal M_{\Gamma}$ only depends on the number of vertices
and edges of $\Gamma$, but not on the incidences. 
The subspace of ``robust'' parameters,
which we introduce later, will depend in a crucial way 
on the incidence structure of the graph.
\end{rem}

\subsection*{Periodic Driving} 
A {\it  driving protocol} is a smooth path
\[
\gamma\: [0,\tau_D]  \to \cal M_{\Gamma} \, ,
\] 
where the real number $\tau_D > 0$ plays the role of driving time. 
When $\gamma(0) = \gamma(\tau_D)$, we can view $\gamma$ as 
a map $C_{\tau_D}\to \cal M_{\Gamma}$, where $C_{\tau_D}$ is the circle of
length $\tau_D$. If in addition the latter map is smooth,
we will say that $\gamma$ is {\it periodic}. When $\tau_D =1$, we 
say that $\gamma$ is {\it normalized}. 

Observe that a periodic driving protocol equivalent to specifying a pair 
\[
(\tau_D,\gamma) \in \Bbb R_+ \times L\cal M_\Gamma
\]
in which $\gamma$ is a normalized periodic driving protocol.
 Here $L\cal M_{\Gamma}$ denotes the free (smooth) loop space
of $\cal M_{\Gamma}$.

\subsection*{The master operator}
Fix a real number $\beta > 0$.
For a given $(E,W) \in \cal M_{\Gamma}$ we can form, for each 
$i \in \Gamma_0$ and $\alpha \in \Gamma_1$, the real numbers
\begin{equation} \label{rates}
g_\alpha = e^{\beta W_\alpha}\, , \qquad \kappa_i = e^{\beta E_i} \, .
\end{equation}
Let $\hat g\: C_1(\Gamma;\Bbb R) \to C_1(\Gamma;\Bbb R)$ be the linear transformation
given by the diagonal matrix whose entries are $g_{\alpha}$.
Similarly, let $\hat \kappa\: C_0(\Gamma;\Bbb R) \to C_0(\Gamma;\Bbb R)$ be given by the diagonal
matrix with entries $\kappa_{i}$.

\begin{defn}[cf.\ {\cite[Eq.~(10)]{CKS1}}] For a given $(\beta, E,W)$, the {\it master operator} is 
defined to be
\begin{equation}
H = -\partial \hat g^{-1}\partial^* \hat \kappa \, ,
\end{equation}
where $\partial^*\: C_0(\Gamma;\Bbb R) \to C_1(\Gamma;\Bbb R)$ is the formal adjoint to
$\partial$.
\end{defn}
In particular, for fixed $\beta$, we can view the master operator as
defining a smooth map 
\begin{equation}
H\: \cal M_{\Gamma}  \to \End_{\Bbb R}(C_0(\Gamma;\Bbb R))\, .
\end{equation}

\begin{rem} \label{rem:master_operator_remarks}
With respect to the inner product on $C_0(\Gamma;\Bbb R)$
defined by $\langle \mathbf u,\mathbf v\rangle_{\hat \kappa} = 
\mathbf u \hat\kappa \mathbf v^t$, the master operator is self-adjoint. 
We infer that the eigenvalues of the master operator are real, and it is also 
easy to see that they 
are non-positive.
When $E = 0 = W$, the master operator is just the graph
Laplacian $-\partial \partial^*$.

The master operator is
also known as the {\it Fokker-Planck operator} to
emphasize its natural interpretation as the discrete analogue of the
Fokker-Planck operator in Langevin dynamics on smooth spaces.

\end{rem}

\begin{rem} For $i,j \in \Gamma_0$, let $S_{ij} = d^{-1}(\{i,j\})$ if $i \ne j$
and let
$T_i = \{\alpha \in \Gamma_1 |\, \{i\} \subsetneq d(\alpha)\}$. Setting $k_{i,\alpha} :=  
g_{\alpha}^{-1}\kappa_{i}$, the matrix entries of the master operator are 
\[
H_{ij} \quad =\quad \left\{ 
\begin{aligned}
   &\sum_{\alpha \in S_{ij} } k_{i,\alpha} \quad & i \ne j \, , \\
 -&\sum_{\alpha \in T_i}  k_{i,\alpha} \quad & i = j\, ,
\end{aligned}
\right.
\]
where the convention is that $H_{ij} = 0$ when $S_{ij}$ is empty, i.e., there
is no edge connecting $i$ and $j$.
In particular, $\sum_{j\in \Gamma_0} H_{ij} = 0$ and $H_{ij} >  0$ for $i \ne j$ (compare \cite[p.\ 101]{vanKampen}). 
\end{rem}

\begin{rem} We offer comments on some distinctions in  terminology between
mathematics and physics.
In the statistical mechanics literature, $\Gamma$ is usually
a simple graph (no multiple edges and no loop edges).
In this case the numbers  $H_{ij}$ are called {\it rates} and
describe a {\it Markov process} on $\Gamma$ with transition matrix $H$
(observe that $H_{ij} = k_{i,\alpha}$ with $d(\alpha) = \{i,j\}$ in this case).
If $X_t$ denotes the state of the process at time $t$, then  
\begin{equation}
H_{ij} = \lim_{\Delta t \to 0} \frac{P(X_{t+\Delta t} = j|X_t = i)}{\Delta t}\, ,  \qquad i \ne j\, ,
\end{equation}
where the numerator appearing on the right denotes the conditional probability of
transitioning to state $j$ at time $t+\Delta t$, given that one
is in state $i$ at time $t$. 

Because of Eq.~\eqref{rates}, the rates satisfy the {\it detailed balance} equation
\begin{equation} \label{detailed_balance}
 H_{ij}\kappa_j  =  H_{ji} \kappa_i\, .
\end{equation}
which states that the net flow of probability from state $i$ to state $j$ is the same
as that from state $j$ to state $i$.
This means that the Markov process is  time reversible \cite[p.\ 109]{vanKampen}.
Conversely, if the process is time reversible, one can show 
that the parameters $\kappa_i$ and $g_{\alpha}$ are, after possibly rescaling, in the form
given by  Eq.~\eqref{rates}.

What we have described above is the notion of continuous time random walk on a graph. This 
is slightly more general than the notion of random walk
considered in the mathematical literature
(cf.\ \cite[Chap. IX]{Bollobas}).  Mathematicians  usually 
define a random walk to 
 be a reversible {\it Markov chain} rather than the more general notion of reversible
 Markov process (the difference
being that for Markov processes, one considers waiting times at the vertices as part of the walk).
\end{rem}

\subsection*{The master equation}
Fix a periodic driving protocol $(\tau_D, \gamma)$, and 
$\beta >0$. Then we have the associated one parameter family of master operators $H(\gamma(t)) 
\in \End_{\Bbb R}(C_0(\Gamma;\Bbb R))$. 
The {\it master equation}  is given by
\begin{eqnarray}
\label{master-equation} \dot {\mathbf p}(t) = \tau_D H(\gamma(t)){\mathbf p}(t)\, . 
\end{eqnarray}
The master equation governs the time evolution of probability: when $\mathbf p(t)$ is
 normalized,  the component $\mathbf p_i(t)$ represents 
the probability density of observing the state $i$ at time $t$.


\subsection*{The Boltzmann distribution}
Suppose $V$ is a finite dimensional real vector space equipped with basis $\cal B$.
If $E\: \cal B\to \Bbb R$ is a function, and $\beta >0$ is a real number, we may form
the normalized linear combination
$$
Z^{-1}\sum_{j \in \cal B} e^{-\beta E_{j}}j \qquad Z\equiv \sum_{j\in \cal B}e^{-\beta E_{j}}\, .
$$
This is called the {\it (normalized) Boltzmann distribution} of the pair $(E,\beta)$.
(In thermodynamics, $\beta$ represents a multiple of
inverse temperature: $\beta = \frac{1}{k_BT}$, where $T$ is the temperature and
$k_B$ is the Boltzmann constant.) The basis ${\cal B}$ identifies $V$ with its
dual space $V^*$, so we are entitled to consider the function 
$E$ as a vector lying in $V$ having
components $E_i$.
Then for fixed  $\beta$, the Boltzmann 
distribution describes a smooth map
$$
{\mathbf \rho}^{\text{B}}\: V \to \Delta[V]\, ,
$$
where $\Delta[V] \subset V$ is the open standard simplex with respect to the
basis $T$ (i.e., this map sends a vector $E$ to its Boltzmann distribution).
We say $E$ is {\it non-degenerate} if there is a unique $j \in T$
such that the $j$-th component $E_j$ of $E$ is minimizing.

\begin{lem} \label{Boltzmann_lemma} Let $f\: [0,1] \to V$ be a
smooth map with the property that $f(t)$ is non-degenerate for every $t\in [0,1]$.
Then
\[
\frac{d}{dt}{\mathbf \rho}^{\text{\rm B}}(f(t))
\]
tends uniformly in $t$ to the zero vector in the low temperature limit
$\beta\to \infty$.
\end{lem}

\begin{proof} As $[0,1]$ is compact, we
only need to verify the statement pointwise, i.e., 
for each $t \in [0,1]$. To avoid clutter we write $E_i := E_i(f(t))$.
Then the $i$-th component of displayed derivative is
\begin{equation} \label{eq:derivative}
\dot{\mathbf \rho}^{\text{B}}_i =  \frac{\sum_j \beta(\dot{E}_j - \dot{E}_i)e^{\beta(E_i - E_j)}}{(\sum_j e^{\beta(E_i - E_j)})^{2}}\, .
\end{equation}

\noindent {\it Case (1): $i$ is the minimizing vertex.}
 In this instance, the denominator of Eq.~\eqref{eq:derivative} is the 
square of $1 + (\sum_{j\ne i} e^{\beta(E_i - E_j)})$, where each $E_i - E_j < 0$. 
Hence the denominator tends to $1$ in the low temperature limit.
 As for the numerator of Eq.~\eqref{eq:derivative}, 
when $i \ne j$, the term $\beta(\dot{E}_j - \dot{E}_i)e^{\beta(E_i - E_j)}$
tends to $0$ and when $i = j$ it is $0$.
So the low temperature limit of  Eq.~\eqref{eq:derivative} is $0$.
\medskip

\noindent {\it Case (2): $i$ isn't the minimizing vertex.} In this 
instance at least one of $E_i - E_j$ is positive and
Eq.~\eqref{eq:derivative} is dominated by $k\beta/e^{c\beta}$ for a suitable choice of 
constants $k$ and $c$ with $c >0$.
The latter tends to zero in the low temperature limit by L'Hospital's rule.
\end{proof}

\subsection*{The Boltzmann distribution for the population space}
 When $\cal B  = \Gamma_0$, we have $V = C_0(\Gamma;\Bbb R)$. 
The Boltzmann distribution in this case describes a smooth map
$$
\mathbf  \rho^{\text{\rm B}} \: \cal M_{\Gamma} \to \bar C_0(\Gamma;\Bbb R)
$$
whose value at $(E,W)$ depends only on
$E$ and $\beta$.
It is not difficult to show that $\mathbf  \rho^{\text{\rm B}}(E,W) \in C_0(\Gamma;\Bbb R)$
 is in the 
null space of the master operator $H(\beta,E,W)$ (compare \cite[p.~101]{vanKampen}).

\section{Current Generation} 
\label{sec:current_generation}
For a periodic driving protocol $(\tau_D,\gamma)$ 
and  $\beta > 0$, 
the {\it instantaneous current} at $t\in [0,1]$ is defined as
\[
\mathbf J(t) = \mathbf J(\beta,\tau_D,\gamma)(t) \,\, := \,\,  \tau_D \hat g^{-1}\partial^{\ast}\hat\kappa {\mathbf \rho}(t) \in C_1(\Gamma;\Bbb R)\, ,
\]
where $\rho(t)$
is the unique periodic solution of the master equation given by 
Proposition \ref{periodic_solution} below
(here we are assuming
that $\tau_D$ is sufficiently large).
Then the continuity equation 
\[
\partial \mathbf J = -\dot{\mathbf \rho}
\] 
is satisfied, 
in which $\mathbf J(t)$ plays the role of probability flux (\cite[p.\ 193]{vanKampen},
\cite{Horowitz-Jarzynski}).

The {\it average current generated per period}  
is 
\begin{eqnarray}
\label{Q-period}
Q(\beta,\tau_D,\gamma) \,\, :=\,\,  \int_{0}^{1} \mathbf J(t) dt\, .
\end{eqnarray}
This expression measures the net flow of probability 
in a single period $[0,\tau_D]$.

\subsection*{Average current in the adiabatic limit}
In the adiabatic limit $\tau_D \to \infty$,
 both $\mathbf J$ and $Q$
can be expressed in terms of  a certain differential $C_1(\Gamma;\Bbb R)$-valued $1$-form $A$.

For each $(E,W) \in \cal M_{\Gamma}$ and $\beta > 0$, 
the negative of the restricted boundary map
\[
-\partial\: \text{im}(\hat g^{-1}\partial^*) \to \tilde C_0(\Gamma;\Bbb R)
\]
is an isomorphism (here $\hat g^{-1}\partial^*\: C_0(\Gamma;\Bbb R) \to C_1(\Gamma;\Bbb R)$ and
$\text{im}(\hat g^{-1}\partial^*) \subset C_1(\Gamma;\Bbb R)$ denotes its image)
Let $\cal L\: \tilde C_0(\Gamma;\Bbb R) \to \text{im}(\hat g^{-1}\partial^*)$
denote the inverse transformation, and 
let $i\: \text{im}(\hat g^{-1}\partial^*) \to C_1(\Gamma;\Bbb R)$ denote
the inclusion.
Then $i\circ {\mathcal L}^{-1}$ defines a homomorphism 
\[
A(\beta,E,W)\: \tilde C_0(\Gamma;\Bbb R) \to C_1(\Gamma;\Bbb R) \, .
\]
For fixed $\beta$ and variable $(E,W)$,  this defines a smooth map
\[
A \: \cal M_{\Gamma}\times \tilde C_0(\Gamma;\Bbb R) \to C_1(\Gamma;\Bbb R)\, .
\]

The proof of the following is immediate.

\begin{lem}
\label{prop:determine-A} The map $A$ is uniquely characterized  by the
following properties: 
\begin{enumerate}
\item  The composition
\[
 {\mathcal M}_{\Gamma} \times \tilde C_0(\Gamma;\Bbb R) \overset{A}\to C_1(\Gamma;\Bbb R) \overset{-\partial}\to C_0(\Gamma;\Bbb R)
\]
coincides with second factor projection, and
\item for all ${\mathbf J}\in H_{1}(\Gamma)$, we have 
\[
\langle{\mathbf J},A\rangle_{\hat{g}}\,\,  =\,\, 0 \, ,
\]
where $\langle {-},{-}\rangle_{\hat{g}}$ is the inner product on 
$C_1(\Gamma;\Bbb R)$ defined  by 
$\langle \mathbf u,\mathbf v\rangle_{\hat{g}} = \mathbf u\hat g\mathbf v^t$.
\end{enumerate}
\end{lem}

\begin{rem}\label{sum-formula-for-A} The operator $A$ defines 
the solution to Kirchoff's theorem
on electrical circuits (see \cite[p.\ 44]{Bollobas}). Property (2) above
amounts to Kirchoff's voltage law with $\hat g$ defining the resistance matrix. 

An explicit formula for $A$ is given as follows: choose a
basepoint $i \in \Gamma_0$.
Given $(W,E) \in \cal M_{\Gamma}$ and $\beta > 0$, define a linear transformation 
$A^e\: C_0(\Gamma;\Bbb R) \to C_1(\Gamma;\Bbb R)$
whose value at basis elements $j \in \Gamma_0$ is
\begin{equation} \label{eq:expression_for_current}
A^e(j) = \sum_{T} Q_{i}^{T,j}\varrho^{\text{B}}_T \qquad j \in \Gamma_0\, ,
\end{equation}
where the sum is over all spanning trees of $\Gamma$. The term 
$Q_{i}^{T,j}$ is the element of $C_1(\Gamma;\Bbb R)$
defined by the signed sum of edges along the unique path from $i$ to $j$ along
$T$, where an edge has sign $+1$ if and only if its orientation coincides with the path.
The term $\varrho^{\text{B}}_T$
is the real number given by the $T$-component of the Boltzmann distribution whose vector space has basis the set of spanning trees of $\Gamma$, where the energy
function is given by $\sum_{\alpha\in T_1} W_{\alpha}$.
Then $A^e$ restricted to the subspace $\tilde C_0(\Gamma;\Bbb R)$ coincides with $A$.
\end{rem}



Given a periodic driving protocol $(\tau_D,\gamma)$ and $\betaÊ>0$,
application of the normalized Boltzmann distribution gives a loop of normalized population vectors
\[
\mathbf  \rho^{B}(\gamma)\: [0,1] \to \bar C_0(\Gamma;\Bbb R)
\]
given by $t\mapsto \mathbf  \rho^{B}(\gamma(t))$.
Taking the time derivative, we obtain a loop of reduced population vectors
\[
\dot{\mathbf \rho}^{\text{\rm B}}(\gamma)\: [0,1] \to \tilde C_0(\Gamma;\Bbb R) \, .
\]
Then application of $A$ to the pair $(\gamma,\dot{\mathbf \rho}^{\text{\rm B}}(\gamma))$
yields a loop of currents 
$$
A(\gamma,\dot{\mathbf \rho}^{\text{\rm B}}(\gamma)) \:[0,1] \to C_1(\Gamma;\Bbb R)
$$
(This procedure describes a smooth map
$L M_{\Gamma} \to LC_1(\Gamma;\Bbb R)$.)

The following is then a straightforward consequence of the definitions combined with
the Adiabatic Theorem \ref{cor:adiabatic_theorem}.

\begin{prop} \label{adiabatic_defn} Let $\beta > 0$ be fixed. Then in the adiabatic limit we have
\[
\lim_{\tau_D\to \infty} \mathbf J(\beta,\tau_D,\gamma)(t) = A(\gamma(t),\dot{\mathbf \rho}^{\text{\rm B}}(\gamma)(t))
\]
and
\[
\lim_{\tau_D\to \infty} Q(\beta,\tau_D,\gamma) = \int_0^1 A(\gamma(t),\dot{\mathbf \rho}^{\text{\rm B}}(\gamma)(t)) dt\, .
\]
\end{prop}

By appealing to Lemma \ref{prop:determine-A}, one sees that 
the image of the adiabatic limit of $Q$ 
is contained in $H_1(\Gamma;\Bbb R)$, so it  defines a smooth map 
\begin{equation}\label{eq:Q-beta}
Q_{\beta}\: LM_\Gamma \to H_1(\Gamma;\Bbb R)\, ,
\end{equation}
where in this notation $Q_\beta(\gamma) := \lim_{\tau_D\to \infty} Q(\beta,\tau_D,\gamma)$.

\begin{rem} As our main results are stated in the adiabatic limit,
there is no loss in pretending, even
before taking the adiabatic limit, that the average current 
is given by the expression on the right-hand side of Proposition \ref{adiabatic_defn}. With this change, the average current is defined without having to refer to either $\tau_D$ or to 
solutions of the master equation.
\end{rem}

\begin{defn} The map 
\[
Q_\beta  \: 
L\cal M_\GammaÊ@>>> H_1(\Gamma;\Bbb R)\, ,
\]
given by Eq.~\eqref{eq:Q-beta}
is called the {\it analytical current map}.
\end{defn}




\section{Good parameters}
\label{sec:good_parameters}

\subsection*{Spanning trees}
For each total ordering $\sigma$ of the set of edges $\Gamma_1$, we may 
define a spanning tree $T_\sigma$ for $\Gamma$ by sequentially removing the links with the highest possible value in the ordering such that the remaining graph remains connected.
Explicitly, let $\alpha_1$ in $\Gamma_1$ be maximal. We discard
 $\alpha_1$ if and only if the graph $\Gamma \setminus \alpha_1 := (\Gamma_0,\Gamma_1 \setminus \{\alpha_1\})$ is connected.
Otherwise, we retain $\alpha_1$. We next consider the edge $\alpha_2$
which is maximal for $\Gamma \setminus \alpha_1$. This is discarded 
if $\Gamma - \{\alpha_1,\alpha_2\}$ is connected.  
Repeating this process, the edges which are retained form
a tree $T_\sigma$.

\begin{defn} The tree $T_{\sigma}$ given by the above procedure
is called the {\it spanning tree associated
with $\sigma$}, or simply the {\it $\sigma$-spanning tree}.
\end{defn}

\begin{figure} 
\includegraphics[scale=.3]{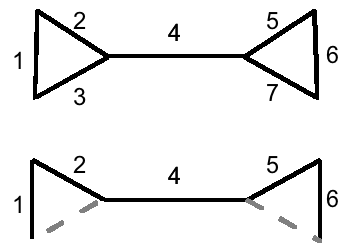}
\captionsetup{name=Fig.\!}
\caption{A graph with a total ordering of its edges and
its associated $\sigma$-spanning tree.}
\label{bow-tie}
\end{figure}

\begin{ex} Consider the graph with total ordering $\sigma$ of its edges depicted
at the top of Fig.~\ref{bow-tie}. The associated $\sigma$-spanning tree is gotten as follows.
If we remove the edge labeled $7$, then the graph is connected, so we discard this edge.
In $\Gamma \setminus 7$, the edge labeled 6 disconnects the graph when it is
removed, so edge 6 is retained. Continuing in this fashion, the all edges but
those labeled 3 and 7 are retained. This results in the spanning tree indicated by 
the path $1\to2\to 4\to 5\to 6$ indicated in the bottom of Fig.~\ref{bow-tie}.
\end{ex}

\begin{ex}\label{sigma-from-W} A total ordering $\sigma$ is determined by a choice of non-degenerate
barrier energies $W\: \Gamma_1 \to \Bbb R$, where $\alpha < \alpha'$ if and only
if $W_\alpha < W_{\alpha'}$. For any such $W$ and
any spanning tree $T$ we introduce the number
\[
w = w(T,W)=\sum_{\alpha \in \Gamma_{1}\setminus T_{1}}W_{\alpha} \, .
\]
When $W$ is understood, we sometimes write $w(T)$ for $w(T,W)$.
\end{ex}

\begin{prop} \label{spanning-tree-prop}
Let $W\:\Gamma_{1} \to \mathbb{R}$ be nondegenerate. 
Let $\sigma$ be the ordering of edges associated with $W$, as in Example 
\ref{sigma-from-W}. Then, for any spanning tree $T \subset \Gamma$ with $T \ne T_{\sigma}$, there is a spanning tree $T' \in \Gamma$, so that $w(T',W) > w(T,W)$.
\end{prop}

\begin{proof} Let $(\alpha_{1},\ldots,\alpha_{k})$ and $(\beta_{1},\ldots,\beta_{k})$ be the elements of $\Gamma_{1}\setminus(T_{\sigma})_{1}$ and $\Gamma_{1}\setminus T_{1}$, respectively, in the decreasing order with respect to $\sigma$, i.e., the barrier energies are decreasing from left to right. Let $j$ be smallest  index such that
\begin{itemize}
\item $\alpha_{j} \ne \beta_{j}$, and 
\item $\alpha_{i}=\beta_{i}$ for $i < j$.
\end{itemize} 
Consider the spanning subgraph $\Gamma' \subset \Gamma$, obtained from $\Gamma$ by withdrawing the edges $\alpha_{1},\ldots,\alpha_{j-1}$, or equivalently $\beta_{1},\ldots,\beta_{j-1}$. By the definition of $T_{\sigma}$, the edge $\alpha_{j}$ is not a bridge of $\Gamma'$ (i.e., its withdrawal does not disconnect the graph), and $W_{\beta_{i}} < W_{\alpha_{j}}$ for $i \ge j$. Let $T^{(1)}$ and $T^{(2)}$ be the two trees obtained from $T$ by withdrawing the edge $\alpha_{j}$. Then there is at least one edge, say $\beta_{s}$, among $\beta_{j},\beta_{j+1},\ldots,\beta_{k}$ that connects $T^{(1)}$ to $T^{(2)}$, since otherwise the edge $\alpha_{j}$ would be a bridge of $\GammaÕ$. Therefore, by replacing the edge $\beta_{s}$ with $\alpha_{j}$ in $T$ results in another spanning tree, denoted $T'$ that obviously satisfies the condition of the proposition, since $W_{\alpha_{j}}>W_{\beta_{s}}$.
\end{proof}

Proposition \ref{spanning-tree-prop} gives an immediate characterization of
 $\sigma$-spanning trees in terms of the function $w({-},W)$. Let $\cal T_{\Gamma}$
 denote the set of spanning trees of $\Gamma$.

\begin{cor} \label{cor:maximum} With $\sigma$ and $W$ as in Proposition \ref{spanning-tree-prop}, the $\sigma$-spanning tree $T_{\sigma}$
is the unique maximizer of the function
 $w({-},W)\: \cal T_{\Gamma} \to \Bbb R$.
\end{cor}

\begin{rem} 
One can restate the last corollary so as to depend only on $\sigma$:
For $\alpha \in \Gamma_1$ set $W_\sigma(\alpha) = k$ if 
$\alpha$ is the $k$-th element in the partial ordering given by $\sigma$.
Now define $\omega(T,\sigma) := w(T,W_\sigma)$. Then by Corollary \ref{cor:maximum},
 $T_\sigma$ 
is the unique maximizer of
$\omega({-},\sigma)\: \cal T_{\Gamma} \to \Bbb R$.
\end{rem}

\begin{rem} There is a useful  alternative characterizing property of the $\sigma$-spanning tree $T_{\sigma}$ associated with a nondegenerate barrier function $W$: for each withdrawn edge 
(i.e., an edge not in $T_{\sigma}$) with $d(\alpha)  = \{i, j\}$, the barrier $W_{\alpha}$  is higher than any of the barriers associated with the edges of the unique embedded path
which connects $i$ to $j$ inside $T_{\sigma}$. For this reason we named $T_{\sigma}$ the  {\it minimal spanning tree} of $W$ in \cite{CKS1}.
\end{rem}

\subsection*{Good parameters and the weak map $\breve q$}
Define an open subset
$$
\breve {\cal M}_{\Gamma} \subset {\cal M}_{\Gamma}
$$
as follows: a pair $(E,W)$ lies in $\breve {\cal M}_{\Gamma}$ if and only if
one of the following conditions hold:
\begin{enumerate}
\item there is only one absolute minimum for $E\: \Gamma_0\to \Bbb R$, or
\item the function $W\: \Gamma_1 \to \Bbb R$ is one-to-one, i.e., the edges are distinguished
by their barrier energies. (In this instance we say that $W$ is {\it nondegenerate.})
\end{enumerate}
We call $\breve {\cal M}_{\Gamma} $ the {\it space of good parameters}.

Let $U$ be the set of $(E,W)$ satisfying the first condition and let $V$ be the set
of $(E,W)$ satisfying the second. Then
\[
\breve {\cal M}_{\Gamma} = U \cup V
\]
where $U$ and $V$ are open. Each connected component of $U$ is defined by specifying 
a vertex $v\in \Gamma_0$, whereas each connected component of $V$ is given by specifying
a total ordering $\sigma$ of $\Gamma_1$. Consequently, we have decompositions into
connected components
\[
U = \coprod_{v} U_v    \qquad \text{ and }\qquad  V = \coprod_{\sigma} V_\sigma\, .
\]
For each vertex $v\in \Gamma_0$, let $B_{1/3}(v)$ be the set of points in $|\Gamma|$
which have distance $< 1/3$ from $v$ in the natural metric on $|\Gamma|$ that
gives every edge a length of $1$.
Let 
$N_v \subset \breve {\cal M}_{\Gamma} \times |\Gamma|$ be the subspace
given by $U_v \times B_{1/3}(v)$. Then the second factor projection
$$
N_v \to U_v
$$
is a homotopy equivalence (it is the cartesian product of $U_v$ with
$B_{1/3}(v)$).
Now set $N_U = \amalg_v N_v$. Then the projection
$N_U \to U$ is also a homotopy equivalence. Call this projection $p_U$.

For a given $\sigma$, we let 
$N_\sigma \subset \breve {\cal M}_{\Gamma} \times |\Gamma|$
be the subset consisting of $V_\sigma \times |T_\sigma|(1/3)$, where $T_\sigma$
is the $\sigma$-spanning tree and
$|T_{\sigma}|(1/3)$ consists of the points of $|\Gamma|$ whose distance to 
$|T_{\sigma}|$ is $< 1/3$. Then the projection $N_{\sigma} \to V_{\sigma}$
is a homotopy equivalence (it is the cartesian product of $V_{\sigma}$ with 
a metric tree). Set $N_{V} = \amalg_{\sigma} V_{\sigma}$. Then the projection
$p_V\: N_V \to V$ is a homotopy equivalence.

Notice that $p_{U}^{-1}(U_v \cap N_\sigma) \subset N_{\sigma}$.
Consequently, if we set
\[
N = N_{U} \cup N_V\, ,
\]
then 
a straightforward application of the gluing lemma \cite{tom_Dieck}
shows that the first factor projection 
\[
p_1\: N \to \breve {\cal M}_{\Gamma}
\] 
is a homotopy
equivalence.

\begin{defn} For good parameters, the weak map $\breve q$ 
is given by
\begin{equation} \label{eq:good_weak_q}
\breve {\cal M}_{\Gamma} @< p_1<{}^\sim <   N @> p_2 >> |\Gamma|\, ,
\end{equation}
where $p_2$ denotes the second factor projection.
\end{defn}



\section{A weak form of the Pumping Quantization Theorem}
\label{sec:PQT}

Recall the decomposition 
\[
\breve {\cal M}_{\Gamma} = U \cup V
\]
of the previous section, where 
$$
U = \coprod_{j}  U_j \, , \qquad V = \coprod_{\sigma} V_{\sigma} \, ,
$$
where $j$ ranges through the elements of $\Gamma_0$ and $\sigma$ ranges through
the set of total orderings of $\Gamma_1$. 

Given a loop $\gamma\in L\breve {\cal M}_{\Gamma}$, it will
be convenient in what follows to think of $\gamma$ 
as a smooth map $C \to \breve {\cal M}_{\Gamma}$, where $C$ 
denotes the circle of radius $1/(2\pi)$. 
Let $I \subset C$ be a closed arc.
The contribution along $I$ to the analytical current map is then given by 
the integral
\[
\int_I \mathbf J ds \in C_1(\Gamma;\Bbb R)
\]
where we parametrize $I$ with respect to arc length. That is, if $I_1,\dots I_k$ is
a simplicial decomposition of $C$ into closed arcs, then
\[
Q_\beta(\gamma) = \sum_{j=1}^k  \int_{I_j} \mathbf J ds\, .
\]
Assume that $\gamma(I) \subset U$; in this instance we say
$C$ is of {\it type $U$}.

\begin{lem}\label{type-U} If $I$ is of type $U$, then in the low
temperature limit the contribution 
along $I$ to  $Q_\beta(\gamma)$ is trivial.\end{lem}

\begin{proof} On $I$ the
function $E \: \Gamma_0 \to \Bbb R$ has a unique absolute minimum $v$.
As $\beta$ tends to $\infty$, the value of Boltzmann distribution
$\rho^{\text{B}}$ restricted 
to arc $I$ tends to $v$. This is because the component of $\ell \in \Gamma_0$ in the Boltzmann
distribution is 
\[
\frac{e^{-\beta E_\ell}}{\sum_i e^{-\beta E_i}}Ê\, \,
\]
and the latter tends to zero on 
if $\ell \ne v$ and one if $\ell = v$ when $E_v$ is the unique minimum.
Consequently, as $\beta$ tends to $\infty$, 
the time derivative $\dot{\mathbf \rho}^{\text{B}}$ tends to zero  on 
$I$ (by Lemma \ref{Boltzmann_lemma}).

The contribution to the current of $\gamma$ along $I$ is given by the integral
\[
 \int_{I} A(\gamma,\dot{\mathbf \rho}^{\text{B}}(\gamma))  ds \, ,
\]
(using Proposition \ref{adiabatic_defn}).
When $\beta$ tends to infinity, this expression tends to zero.
\end{proof}

Now consider a closed arc $I  = [a,b]\subset C$ with endpoints $a,b$ such that
\begin{itemize}
 \item $\gamma(I) \subset V$, and
\item $\gamma(\partial I) \subset U$.
\end{itemize}
In this instance we say $I$ is of {\it type $V$}. 

\begin{lem} \label{lem:current-formula} If $I$ is of type $V$, 
then in the low temperature limit the contribution along $I$ to
$Q_\beta(\gamma)$ is an element of
$C_1(\Gamma;\Bbb Z)$. 
\end{lem}

\begin{proof} Fix a basepoint $i \in \Gamma_0$, $(E,W) \in {\cal M}_{\Gamma}$
and $\beta >0$. Recall from Remark \ref{sum-formula-for-A} the formula
\[
A^e(j) = \sum_{T} Q_{i}^{T,j}\varrho^{\text{B}}_T \qquad j \in \Gamma_0\, ,
\]
where $A^e \: C_0(\Gamma;\Bbb R) \to C_1(\Gamma;\Bbb R)$ restricts to $A$
on  $\tilde C_0(\Gamma;\Bbb R)$. Here $\varrho^{\text{B}}_T$ is
the $T$-component of the Boltzmann distribution for the 
 vector space whose basis is the set of spanning trees of $\Gamma$.
Recall also that the instantaneous current
$\mathbf J(t)$ is defined as $A(\dot{\rho}^{\text{B}})$, where $\rho^{\text{B}}$
in this case denotes the Boltzmann distribution for $C_0(\Gamma;\Bbb R)$. 
Hence, inserting $\dot\rho^{\text{B}}$ into the expression for $A^e$ gives
\begin{equation} \label{eq:J-formula}
\mathbf J(t) = \sum_{T,j} Q_{i}^{T,j}\varrho^{\text{B}}_T \dot\rho^{\text{B}}_j\, .
\end{equation} 
Recall that the average current is given by $\int_0^1 \mathbf J(t)$.
In particular, the contribution along $I$ is given by $\int_{c}^{d} \mathbf J(t) dt$
(where we are parametrizing $I$ with respect to arc  lengthÊ  and the limits of integration
come from the parametrization). 
Hence, integrating both sides of the last display, we obtain
\begin{equation}
Q = \sum_{T,j} Q_{i}^{T,j}\int_c^d\varrho^{\text{B}}_T \dot\rho^{\text{B}}_j\, .
\end{equation}
Since $Q_{i}^{T,j}$ is an integer valued $1$-chain, it will suffice to prove that
\[
\int_c^d\varrho^{\text{B}}_T \dot\rho^{\text{B}}_j
\]
tends to an integer as $\beta$ tends to $\infty$.
Using integration by parts, we may rewrite this expression as
\[
\varrho_T^{\text{B}} \rho_j^{\text{B}}|_{c}^{d} \,\, -\,\, 
\int_c^d \dot \varrho^{\text{B}}_T {\mathbf \rho}^{\text{B}}_j\, .
\] 
By Lemma \ref{Boltzmann_lemma}, $\dot {\mathbf \rho}^{\text{B}}_T$ tends to zero
as $\beta\to \infty$. Hence the constribution to the current along $I$ in the
low temperature limit is determined by the value of 
${\mathbf \varrho}_T^{\text{B}} {\mathbf \rho}_T^{\text{B}}|_{a}^{b}$. Since $\gamma(\partial I) \subset U$, 
we deduce by the argument of Lemma \ref{type-U} that the low temperature limit of ${\mathbf \varrho}_T^{\text{B}} {\mathbf \rho}_T^{\text{B}}|_{a}^{b}$ is an 
integer.
\end{proof}

As an immediate corollary, we obtain a weak version of Theorem \ref{thm:strongPQT}:

\begin{thm}[Weak Quantization] If $\gamma\in L\breve M_{\Gamma}$, then the low temperature
limit $\lim_{\beta\to \infty} Q_\beta(\gamma)$ is defined and lies in the integral lattice
$H_1(\Gamma;\Bbb Z) \subset H_1(\Gamma;\Bbb R)$.
\end{thm}

\section{The Discriminant Theorem and Robust Parameters}
\label{sec:discriminant_theorem}
Set
$$
\breve {\cal D} := {\cal M}_{\Gamma} \setminus \breve {\cal M}_{\Gamma}.
$$
We will first show that the one-point compactification $\breve {\cal D}^+$ 
has the structure of a regular CW complex.
By definition, $\breve {\cal D}$ is the subspace of ${\cal M}_{\Gamma}$ consisting of
pairs $(E,W)$ such that $E\: \Gamma_0 \to \Bbb R$ has more than one absolute minimum and $W\: \Gamma_1 \to \Bbb R$ is not one-to-one. 

\begin{defn} A {\it height function} for $\Gamma$ is a pair of functions
$$
h_0:\Gamma_0\to \{1,2\} \, , \qquad h_1\: \Gamma_1 \to \{1,\dots,n\}\, ,
$$
where 
\begin{itemize}
\item $n > 0$ is an integer,
\item $h_0^{-1}(1)$ is non-empty, and
\item $h_1$ is surjective.
\end{itemize}
We write $h := (h_0,h_1)$.\end{defn}

Height functions arise in the following situation.

\begin{ex} \label{ex:height-function}
Given $(E,W) \in \breve {\cal D}$, 
we write 
$h_0(i) = 1$ if and only if $i$ is a minimum for $E$, and otherwise
we set $h_0(i) =2$. 
We define  $h_1\: \Gamma_1 \to \{1,\dots,n\}$ to be the unique surjective function
characterized by
\begin{itemize}
\item $h_1(\alpha_1) \le h_1(\alpha_2)$ if and only if $W_{\alpha_1} \le W_{\alpha_2}$,
and 
\item $h_1(\alpha_1) = h_1(\alpha_2)$ if and only if $W_{\alpha_1} = W_{\alpha_2}$.
\end{itemize}
The pair $h := (h_0,h_1)$ is then a height function for $\Gamma$.
\end{ex}

Given a height function $h = (h_0,h_1)$ we define
$$
C(h)
$$
to be the set of all $(E,W) \in \breve {\cal D}$  
whose associated  height function is $h$, as in Example \ref{ex:height-function}. 
Then $\breve{\cal D} = \coprod_{h} C(h)$
as sets. Note that $C(h)$ is non-empty if and only if $h_0^{-1}(1)$
has more than one element and $h^{-1}_1(k)$ has more than one element for
some $k$. Let $D(h)$ denote the closure of $C(h)$
in $\breve {\cal D}$. 

\begin{prop} Assume that $C(h)$ is non-empty. 
Then the one-point compactification $D(h)^+$ is homeomorphic to a disk of dimension $m+n$, where
$m =1+|\Gamma_0 \setminus h^{-1}_0(1)|$
and $n$ is as above. 
\end{prop} 

\begin{proof} The space $C(h)$ coincides with the cartesian 
product $C_0(h) \times C_1(h)$,
where $C_0(h)$ consists of those $E:\Gamma_0 \to \Bbb R$ 
associated with $h_0$  
and $C(h_1)$ consists of $W:\Gamma_0 \to \Bbb R$ associated with $h_1$. Likewise $D(h)$ coincides with the cartesian
product $D(h_0) \times D(h_1)$.
 
Note that $D(h_0)$ is canonically homeomorphic to the space $P_m$ consisting of $m$-tuples or real numbers
$$
(x_1,x_2,\dots, x_m)
$$
such that $x_1 \le x_k$ for all $k >1$. The homeomorphism is given by mapping
$E$ to the map $\Gamma_0/{\sim} \to \Bbb R$, where
where
$\sim$ denotes the equivalence relation on $\Gamma_0$ defined by 
$i \sim j$ if and only of $h_0(i) = 1 = h_0(j)$. As $\Gamma_0/{\sim}$
has cardinality one more than the number of non-minima of $E$, 
such functions are identified with
$m$-tuples of real numbers. The operation  
$$
(z_1,\dots,z_m) \mapsto (z_1,z_2+z_1,\dots z_m+z_1)
$$
defines a homeomorphism to $P_m$ from the space of 
of $m$-tuples $(z_1,\dots,z_m)$ such that $z_1 \in \Bbb R$ and 
$z_k \ge 0$ for all $k > 1$. It is easy see that the one-point compactification
of the latter space is homeomorphic to $D^m$. Hence, $D(h_0)^+$ is
homeomorphic to $D^m$.

Similarly, the space $D(h_1)$ is canonically homeomorphic the space $Q_n$ 
consisting of $n$-tuples or real numbers
$$
(x_1,x_2,\dots, x_n)
$$
such that $x_1 \le x_2 \le \dots x_n$.
The operation
$$
(z_1,\dots,z_n) \mapsto (z_1,z_1+z_2,\dots, \sum_{i=1}^n z_i)
$$
defines a homeomorphism to $Q_n$ from the space of $n$-tuples 
$(z_1,\dots,z_n)$ for which $z_1 \in \Bbb R$ and $z_i \ge 0$ for $i >1$.
The one-point compactification of the latter space
is identified with $D^n$. Consequently, $D(h_1)^+$ is homeomorphic
to $D^n$.   Finally,
\begin{align*}
D(h)^+ &= (D(h_0) \times D(h_1))^+ \\
            &= D(h_0)^+ \smsh D(h_1)^+ \\
            &\cong D^m \smsh D^n \\
            &= D^{m+n}\, .
\end{align*}
\end{proof}

\begin{cor} $\breve {\cal D}^+$ has the structure of a regular CW complex of
dimension $d-2$, where $d$ is the cardinality of $\Gamma_0 \amalg \Gamma_1$.
\end{cor}

\begin{rem} With the exception of the
point at $\infty$, the open cells of $\breve {\cal D}^+$ are given by the $C(h)$
where $h$ varies over the height functions for $\Gamma$. This is because
$C(h)$ is the interior of $D(h)^+$.

A top dimensional cell $C(h)$ of $\breve {\cal D}^+$ 
is given by a height function $h$ in which
\begin{itemize}
\item $h_0^{-1}(1)$ has precisely two elements;
\item there is a $k \le n$ such that $h_1^{-1}(k)$   
has precisely two elements, and if $i \ne k$ we have $h^{-1}_1(i)$ is a singleton
for $1\le i \le n$.
\end{itemize}
\end{rem} 
 
\subsection*{Robust parameters}
Alexander duality applied to the inclusion $\breve {\cal D}^+ \subset \cal M_{\Gamma}^+ = S^d$
yields an isomorphism 
$$
H^{d-2}(\breve {\cal D}^+ ) \cong H_1(\breve {\cal M}_{\Gamma}) \, .
$$
Let $C^{d-2}(\breve {\cal D}^+)$ be the cellular cochain complex of $\breve {\cal D}^+$
over the integers in degree $d-2$. This is the free abelian group with basis 
given by the set of $(d-2)$-cells of $\breve {\cal D}^+$. We consider the composition
\begin{equation} \label{dangerous}
C^{d-2}(\breve {\cal D}^+) \to H^{d-2}(\breve {\cal D}^+ )  \cong H_1(\breve {\cal M}_{\Gamma}) 
\overset {q_*}\to H_1(\Gamma;\Bbb R)
\end{equation}
where $q\: \breve {\cal M}_{\Gamma} \to |\Gamma|$ is the weak map defined in 
Eq.~\eqref{eq:good_weak_q} above.
This homomorphism is naturally identified with a $H_1(\Gamma;\Bbb R)$-valued chain
$$
\phi \in C_{d-2}(\breve {\cal D}^+;H_1(\Gamma;\Bbb R))\, ,
$$ 
and it is trivial to check that $\phi$ is a cycle. For $a\in C^{d-2}(\breve {\cal D}^+)$, let 
$\langle a ,\phi\rangle \in H_1(\Gamma;\Bbb R)$ denote the effect of applying the homomorphism
\eqref{dangerous} to $a$.

\begin{defn} A $(d-2)$-cell $C(h)$  of $\breve {\cal D}^+$ is said to be {\it essential} if 
$\langle C(h),\phi\rangle \in H_1(\Gamma;\Bbb R)$ is non-trivial, where
we consider $C(h)$ as an element of $C^{d-2}(\breve {\cal D}^+)$. 
A cell $C(h)$ of $\breve {\cal D}^+$ of any dimension
is {\it inessential} if it is not contained in the closure of an
essential $(d-2)$-cell.
\end{defn}

Define $\check {\cal D}$ to be the closure of the union of the essential $(d-2)$-cells
of $\breve{\cal D}$.

\begin{lem}  $\check {\cal D}^+$ is a subcomplex of $\breve{\cal D}^+$.
\end{lem}

\begin{proof} It is enough to show that  the union of any collection of
top dimensional closed cells of $\breve {\cal D}^+$ forms a subcomplex. 
Let $D(h)^+$ be any closed cell of $\breve{\cal D}$.
Suppose $x$ lies in the boundary of $D(h)$.
Then $x$ lies in a unique (open) $j$-cell $C(h')$.
It is straightforward to check  that $D(h')\subset D(h)$. Consequently,
the boundary of $D(h)^+$ is a union of lower dimensional cells, each of
these having boundary a union of lower dimensional cells and so on. In particular,
$\breve {\cal D}^+$ is a union of interiors of certain cells, and this union is closed.
Hence it is a subcomplex.
\end{proof}

\begin{defn}
Set
$$
\check {\cal M}_{\Gamma}  := {\cal M}_{\Gamma} \setminus \check {\cal D}\, .
$$
Then we have an inclusion $\breve {\cal M}_{\Gamma}  \subset \check {\cal M}_{\Gamma}$. We
call $\check {\cal M}_{\Gamma}$ the space of {\it robust parameters}.
\end{defn}

\begin{figure}
\includegraphics[scale=.5]{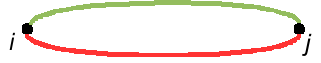}
\captionsetup{name=Fig.\!}
\caption{A two state system.}
\label{2-state}
\end{figure}

\begin{ex} \label{ex:2-state-graph} Let $\Gamma$ denote the graph displayed in Figure \ref{2-state}. In this case, the vector space of parameters $\cal M_{\Gamma}$ is identified with $\Bbb R^4$ and
the discriminant $\breve {\cal D}$ is identified with the diagonal 
inclusion $\Bbb R^2 \subset \Bbb R^4$
given by $(x,y) \mapsto (x,x,y,y)$. Taking the one-point compactification of
this inclusion, we obtain an inclusion $S^2 \subset S^4$ that is identified
with $\breve {\cal D}^+ \subset {\cal M}_{\Gamma}^+$. The complement of
this inclusion has the homotopy type of $S^1$. 
 Consequently, there is a homotopy equivalence
$$
\breve {\cal M}_{\Gamma} \, \simeq S^1 \, .
$$
In fact, one can make
this identification precise using the loop $\gamma$
given by the length one periodic driving protocol $\gamma(t) = (\cos 2\pi t,0,\sin 2\pi t,0)$
(one verifies this by showing that linking number of $\gamma$ with $S^2 \subset S^4$
is $\pm 1$).

Consequently,
the first homology group of $\breve {\cal M}_{\Gamma}$  is generated  by
the homology class $[\gamma]$.
The effect of ${\breve q}_*\: H_1(\breve {\cal M}_\Gamma) \to H_1(\Gamma)$ on this loop
is to produce a homology generator of $H_1(\Gamma;\Bbb R) \cong \Bbb Z$ (see \cite{CKS1} for details),
so the weak map ${\breve q}\: \breve {\cal M}_\Gamma \to |\Gamma|$ is a weak homotopy equivalence.
In particular, the unique $2$-cell of $\breve {\cal D}^+ \cong S^2$ is  essential, and we infer in this instance $\check {\cal M}_{\Gamma} = \breve {\cal M}_{\Gamma}$.
\end{ex} 

\begin{rem}\label{rem:non-triviality} With little difficulty,
Example \ref{ex:2-state-graph} can be generalized
to show that ${\breve q}_*\: H_1(\breve {\cal M}_\Gamma) \to H_1(\Gamma)$  is non-trivial whenever $\Gamma$
has non-trivial first Betti number. We omit the details.
\end{rem}

\begin{figure}
\includegraphics[scale=.5]{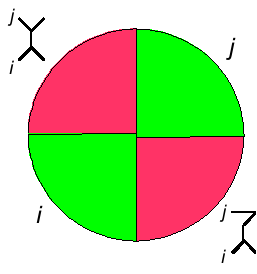}
\captionsetup{name=Fig.\!}
\caption{A two dimensional disk which meets a $(d-2)$-cell of $\breve {\cal D}$
transversely at its center. The green sectors are
contained in $U$ and the red sectors are contained in $V$. In this example,
the current is represented by a non-trivial cycle so the $(d-2)$-cell is essential.}
\label{sectors} 
\end{figure}

\subsection*{The weak map $\hat q$} 
Let $C(h)$ denote a $(d-2)$-cell of $\breve{\cal D}^+$, and let 
$x\in C(h)$ be its center. We choose a small closed $2$-disk meeting $C(h)$ normally
at $x$. The boundary of this disk represents a periodic driving protocol $\gamma$ in the space of good parameters. As an illustration, we have depicted such a disk in Figure \ref{sectors},
 associated with the graph $\Gamma = $ 
$  > \!\!\!{-}\!\!\lhd$.

If the disk is sufficiently small then it is partitioned into  four  regions, the interior of each either contained in  $U$ or $V$, which is shown in Figure \ref{sectors} as 
green and red respectively.\footnote{One may justify this as follows: at $x$ there is a unique
pair of vertices $i$ and $j$ and unique pair of edges $\alpha,\beta$ in which $E_i = E_j$ 
are minimizing and $W_{\alpha} = W_{\beta}$. A generic infinitesimal perturbation into $\breve {\cal M}_{\Gamma}$ of 
thse values will give one of the following inequalities: $E_i > E_j$, $E_{i} < E_j$,
$W_{\alpha} > W_{\beta}$ or $W_{\alpha} < W_{\beta}$. These inequalities correspond to the four sectors.}

Each green sector corresponds to a vertex $i$ or $j$ of $\Gamma$, 
each giving a minimum for  $\{E_k\}_{k \in \Gamma_0}$. Likewise, each
red sector corresponds to a preferred maximal spanning
tree, and these are indicated next to each such sector together with the 
location of the vertices $i,j$. The topological current associated with the driving protocol $\gamma$ is defined by the cycle 
obtained by joining together two paths connecting $i$ with $j$ (cf.\ \S8). Each path is defined
by moving along one of the spanning trees starting at $i$ and terminating at $j$
(for example, in Figure \ref{sectors} one obtains the cycle $\lhd$, which is nontrivial).
In particular, the cell $C(h)$ is inessential if and only if this cycle is trivial.
However, even more is true: since each path used to form the cycle is embedded, it is straightforward to
check that the loop in $\Gamma$ given by gluing these two paths together
 is null homotopic if and only if the two paths coincide. Consequently, if
$\sigma$ is inessential then the weak map $\breve q$ extends to the space given by 
attaching a two cell to $\breve {\cal M}_{\Gamma}$ along $\gamma$. 
If we repeat this construction for every inessential $(d-2)$-cell, we obtain a space
which is homotopy equivalent to the space of robust parameters $\check {\cal M}_{\Gamma}$.
We have therefore shown that the weak map $\breve q$ admits an extension to the space of robust parameters. We denote this extension by $\hat q$.

However, in order to prove the Pumping Quantization Theorem, it will more convenient to 
give a concrete extension of the weak map $\breve q$ to all of $\check {\cal M}_{\Gamma}$, rather than just a model for the extension up to homotopy. This construction will be described in the next section.

\begin{figure} 
\includegraphics[scale=.8]{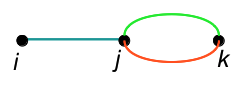}
\captionsetup{name=Fig.\!}
\caption{A graph with three vertices and three edges. }
\label{3-vertex}
\end{figure}

\begin{ex} In the  graph depicted by Figure \ref{3-vertex},
the edge containing vertices $i$ and $j$
is in every maximal spanning tree. 
Therefore, the discriminant $\breve {\cal D}$ for this graph 
has an inessential $4$-cell.
\end{ex}

\begin{defn} A {\it barrier resolution} of a height function $h = (h_0,h_1)$
is a bijection $r\:\Gamma_1\to  \{1,\dots,|\Gamma_1|\}$, where
$|\Gamma_1|$ is the cardinality of $\Gamma_1$, such that
$h_1(\alpha) < h_1(\beta)$ implies $r(\alpha) < r(\beta)$
for all $\alpha,\beta\in \Gamma_1$.
\end{defn}

A barrier resolution $r_h$ of $h$ enables one to associate a total ordering $\sigma_{r}$
of the set $\Gamma_1$. We saw in \S \ref{sec:good_parameters} how to obtain a spanning tree $T_{\sigma_r}$ for
$\Gamma$ associated with $\sigma_r$. Let 
\[
F_h = \bigcap_{r} T_{\sigma_r}\, ,
\] 
where the intersection is indexed over the set of barrier resolutions
$r$ of $h$. Then $F_h$ is a forest, i.e., a  (possibly empty) disjoint union of trees.

\begin{prop}\label{prop:inessential-cell}
A cell $C(h)$ is inessential if and only if all 
elements of $h_0^{-1}(1)$
belong to the same connected component of $F_h$. \end{prop}

\begin{proof}
Consider a cell $C(h)$ that satisfies the condition that all $i \in h_0^{-1}(1)$ belong to the same connected component $T_h$ of the forest associated with $C(h)$. 
 Then $T_{h}$ is a tree. Consider an arbitrary top dimensional cell $C(\hat{h})$, so that $C(h)\subset D(\hat{h})$. The cell $C(\hat{h})$ is uniquely identified by two distinct vertices $i,j\in h_{0}^{-1}(1)$, and two distinct edges $\alpha$ and $\beta$ with $h_{1}(\alpha)=h_{1}(\beta)$, together with a height function $h_{1}$ of height $|\Gamma_1|-1$ with $\hat{h}_{1}(\alpha)=\hat{h}_{1}(\beta)$, i.e., $W_{\alpha}=W_{\beta}$ is the only edge degeneracy in $C(\hat{h})$. Then we have $\hat{h}_{0}(k)=1$ for $k=i,j$ and $\hat{h}_{0}(k)=2$, otherwise. Consider a point $(E,W)\in C(\hat{h})$, and let $\gamma_{\hat{h}}$ be a small loop that goes around $({E},{W})$ without intersecting $C(\hat{h})$, i.e., staying within the good parameter space $\breve{{\cal M}}_{\Gamma}$. Using the notation of Eq.~\eqref{eq:expression_for_current}, we have for the topological current
\begin{eqnarray}
\label{current-cell-1} {Q}(\gamma_{\hat{h}})={Q}_{i}^{T_{\tilde{h}(\alpha)},j} -{Q}_{i}^{T_{\tilde{h}(\beta)},j} \, ,
\end{eqnarray}
where $\tilde{h}_{\alpha}$ and $\tilde{h}_{\beta}$ are the two possible barrier resolutions of $\hat{h}$ with $\tilde{h}_{\alpha}> \tilde{h}_{\beta}$ and $\tilde{h}_{\alpha}< \tilde{h}_{\beta}$, respectively. Then $\tilde{h}_{\alpha}$ and $\tilde{h}_{\beta}$ are also barrier resolutions of $h$, and, therefore, $T_{h}\subset T_{\tilde{h}_{\alpha}}$ and $T_{h}\subset T_{\tilde{h}_{\beta}}$, which implies ${Q}_{i}^{T_{\tilde{h}(\alpha)},j} ={Q}_{i}^{T_{h},j} ={Q}_{i}^{T_{\tilde{h}(\beta)},j}$. Therefore, ${Q}(\gamma_{\hat{h}})=0$, due to Eq.~(\ref{current-cell-1}), so that $C(h)$ is inessential.

To prove the converse,  let $C(h)$ be any inessential cell. For an arbitrary pair of edges $\alpha$ and $\beta$ let $\hat{h}_{1}\: \Gamma_1 \to \{1,\dots, |\Gamma_1| -1\}$ 
be any surjection such that $\hat{h}_{1}(\alpha)=\hat{h}_{1}(\beta)$. Then there are two possible barrier resolutions $\tilde{h}_{\alpha}$ and $\tilde{h}_{\beta}$ of $\hat{h}$, as described above. For arbitrary distinct vertices $i$ and $j$ let $\hat{h}_0\: \Gamma_0 
\to \{1,2\}$ be the function defined by
\[
\hat{h}_{0}(k) = \left\{ \begin{array}{rl}
 1 &\mbox{ if } k = i \mbox{ or } j, \\
 2 &\mbox{ otherwise.}
       \end{array} \right.
\]
Then $\hat h = (\hat{h}_0,\hat{h}_1$) is a height function, and $C(\hat{h})$ is top dimensional.
If $C(\hat{h})$ is inessential then ${Q}(\gamma_{\hat{h}})=0$ and Eq.~(\ref{current-cell-1}) implies that the minimal paths that connect $i$ to $j$ along the spanning trees $T_{\tilde{h}_{\alpha}}$ and $T_{\tilde{h}_{\beta}}$ are identical. Furthermore, if 
$D(\hat{h})\supset C(h)$, then $C(\hat{h})$ is inessential.
Consequently, the set consisting of the minimal paths that connect $i$ to $j$ inside the spanning trees $T_{\tilde{h}}$ associated with all barrier resolutions $\tilde{h}$ of $h$ consists of a single element. Denote this unique path by $l_{h,ij}$.  Then for any two distinct vertices $i,j\in h_{0}^{-1}(1)$ the path $l_{h,ij}$ belongs to all spanning trees $T_{\tilde{h}}$, and therefore the forest associated with $C(h)$. We infer that all vertices that belong to $h_{0}^{-1}(1)$ belong to the same connected component of the forest. 
\end{proof}

\section{The weak map $\check q$}
\label{sec:weak_map}
Proposition~\ref{prop:inessential-cell} shows that to any inessential cell $C(h)$ one can assign, in a preferred way, a tree that contains $h_{0}^{-1}(1)$. This tree will be denoted $T_{C(h)}$ and referred to as the tree associated with the inessential cell $C(h)$. We now cover each open cell $C(h)$ with an open set $Y_{h}\supset C(h)$ that consists of all $({E},{W})\in {\cal M}_{\Gamma}$ that is characterized by the property that if $h_{0}(i)< h_{0}(j)$ and $h_{1}(\alpha)< h_{1}(\beta)$, then $E_{i}< E_{j}$ and $W_{\alpha}< W_{\beta}$, respectively, for all $i,j\in \Gamma_{0}$ and $\alpha,\beta\in \Gamma_{1}$. Setting $Y=\bigcup_{h}Y_{h}$, where the union goes over all inessential cells we obtain an open cover
\begin{eqnarray}
\label{open-cover-robust} \check{{\cal M}}_{\Gamma}=U\cup V \cup Y
\end{eqnarray}
of the space of robust parameters.

For a given height function $h=(h_{0},h_{1})$ whose cell $C(h)$ is inessential, we let 
\[N_{h}\subset \check{{\cal M}}_{\Gamma}\times |\Gamma|
\] 
be the subspace given by  $Y_{h}\times |T_{C(h)}|(1/3)$, where
$|T_{C(h)}|(1/3)$ is the open regular neighborhood of $|T_{C(h)}|$ given by adjoining half open subintervals of length $1/3$ that correspond to those edges not in
$T_{C(h)}$ but which contain a vertex of it.

Then the projection $N_{h}\to Y_{h}$ is a homotopy equivalence. We further define $N_{Y}=\bigcup_{h}{N}_{h}$, and if we set
\begin{eqnarray}
\label{define-check-N} \check{N}=N_{U}\cup N_{V} \cup N_{Y}\, \subset 
\check{{\cal M}}_{\Gamma}\times |\Gamma|\, ,
\end{eqnarray}
then we have a diagram
\begin{eqnarray}
\label{weak-map-extend} 
\check {{\cal M}}_{\Gamma}@< {p_1} << \check{N} @> {p_2} >> |\Gamma|
\end{eqnarray}
whose arrows are given by the first and second factor projections respectively.
A straightforward application of the gluing lemma which we omit shows that the  projection map
$p_1$ is a homotopy equivalence. We infer that the Eq.~\eqref{weak-map-extend} describes a
weak map which we sometimes write as  
\begin{eqnarray}
\label{notation-weak-map-extend} 
\check{q}\: \check {{\cal M}}_{\Gamma} \to |\Gamma|\, .
\end{eqnarray}
By construction, the restriction of $\check q$ to $\breve {{\cal M}}_{\Gamma} $
coincides with $\breve q$.

For the sake of completeness we now
sketch a proof that $\check q$ coincides up to homotopy with 
the  extension $\hat q$ of $\breve q$
that was described in the previous section.

\begin{lem}  The homotopy class of the weak map $\check q$  
coincides with the homotopy class of the weak map $\hat q$ given by 
gluing in 2-cells.
\end{lem}

\begin{proof} The result will follow from the existence of a commutative diagram
\[
\xymatrix{
\hat {\cal M}_{\Gamma}\ar[dd] & \hat N \ar[l]_{p_1}\ar[dd] \ar[rd]^{p_2} \\
&& |\Gamma|\\
\check {\cal M}_{\Gamma} & \check N \ar[l]^{p_1}\ar[ur]_{p_2} \ar[ru]
}
\]
in which the vertical maps are homotopy equivalences and the left horizontal
maps (denoted $p_1$ in each case) are also homotopy equivalences.  The bottom maps
labelled $p_1$ and $p_2$ comprise the weak map $\check q$. The top maps labelled
$p_1$ and $p_2$ define the extension $\hat q$.
The space $\hat {\cal M}_{\Gamma}$ is obtained by attaching suitable two cells to
$\breve {\cal M}_{\Gamma}$; it comes equipped with a decomposition
\[
\hat {\cal M}_{\Gamma} = U \cup V \cup \hat Y
\]
where $\hat Y$ consists of the set of inessential $2$-cells, each which is labelled as 
$\hat D_h$, where $h$ ranges over height functions that define a top dimensional inessential cell. Similarly, we define
$\hat N \subset \hat {\cal M}_{\Gamma} \times |\Gamma|$ to be 
the union of $N_h = \hat Y_h \times |T_{C(h)}|(1/3)$. Then the projections
onto each factor explicitly define the extension of $\breve q$ described in \S\ref{sec:discriminant_theorem}.

The vertical maps in the diagram are given as follows. The homotopy equivalence
$\hat {\cal M}_{\Gamma} \to \check {\cal M}_{\Gamma}$ is given by the identity on 
$\breve {\cal M}_{\Gamma}$ and by   mapping
each closed 2-cell $\check D_h$ that is attached to $\breve {\cal M}_{\Gamma}$
homeomorphically to a small closed 2-cell $\check D_h$ that intersects $C(h)$ 
transversely at its center
(the boundary of $\check D_h$ is prescribed to having linking number +1 with $C(h)$).
A similar argument which we omit defines the homotopy equivalence 
$\hat N \to \check N$.
\end{proof}

\begin{rem}\label{rem:tree_notation} For the proof of the Pumping Quantization Theorem,
it will be convenient to put the open sets $U_{j}$, $V_{\sigma}$, and $Y_{h}$ on equal notational footing. This can be done by changing the notation to 
\[V_{\sigma}: =Y_{h}\, ,\]
where $h = (h_0,h_1)$ is
such that $h_{0}(k)=2$ for all $k\in \Gamma_{0}$, and $h_{1}:\Gamma_{1}\to \{1,\ldots,|\Gamma_{1}|\}$ is given by $h_1$ 
is defined in the obvious sense by the total ordering $\sigma$.
Similarly, we set \[U_{j}:=Y_{h}\, ,\] where in this instance $h = (h_0,h_1)$
is given by $h_{0}(k)=1$ for $k=j$ and $h_{0}(k)=2$ otherwise, and $h_{1}$ is the function with constant value $1$. Note that these notational changes necessitate a more flexible notion of height function, which we will call an {\it extended height function}.

Note that the tree $T_{C(h)}$ that corresponds to $V_{\sigma}$ is just the $\sigma$-spanning tree $T_{\sigma}$, whereas $T_{C(h)}$ that corresponds to $U_{j}$ consists of the single vertex $j$ and no edges. With the above extension Eq.~(\ref{define-check-N}) reads 
$\check{N}=N_{Y}$ with $N_{Y}=\bigcup_{h}\check{N}_{h}$, where the union indexed over the  set of extended height functions. 
\end{rem}

\section{The Representability Theorem}
\label{sec:representability}
\begin{proof}[Proof of Theorem \ref{thm:representability}]
It is clear that
$
L\check {\cal M}_{\Gamma} \subset \check L{\cal M}_{\Gamma}
$,
so it suffices to prove the reverse inclusion. 
The proof will be by contradiction. Let $\gamma \in \check L{\cal M}_{\Gamma}$.
The idea is to modify $\gamma$ along a small arc in such a way that the value of the current changes.

 Suppose there is an $s\in [0,1]$ such that $y = \gamma(s) \in \check{\cal D}$.
 Then $y$ is in the closure of an essential cell. Let $\epsilon >0$ be small.
  Choose a point $x$ in the interior of this cell such that $|y-x|\le \epsilon$ with respect
  to a choice of norm on ${\cal M}_\Gamma$.
Let $V$ be an open neighborhood of $\gamma$ in $L{\cal M}_{\Gamma}$
on which $Q$ is well-defined and constant.

If  $\epsilon $ is sufficiently small, 
we can construct a smooth loop $\omega \in V$ which coincides with $\gamma$ off of
$(s-\epsilon,s+\epsilon)$, and inside this neighborhood
$\omega$ winds once around a small disk $D$ meeting the essential cell 
transversely at $x$ in such a way that $D\setminus x\subset \check {\cal M}_{\Gamma}$ and 
$Q(\partial D)$ is nontrivial. 
Then topologically, $\omega$ is a loop obtained by concatenating $\gamma$ with $\partial D$.
As the current $Q$ is additive, 
we find that $Q(\omega) = Q(\gamma) + Q(\partial D)$. Consequently, $Q(\omega)\neq Q(\gamma)$. 
This contradicts the assumption that $Q$ is  constant
on $V$. \end{proof}

\section{The Pumping Quantization and Realization  Theorems}
\label{sec:strongPQT}
Consider a closed arc $I=[a,b]\subset C$ of our unit length circle $C$ such that $\gamma(I)\subset Y_{h}$ for some $h$. Obviously, for $c=a,b$ there is a well-defined low-temperature limit $\rho^{(c)}=\lim_{\beta\to\infty}\rho^{B}(\gamma(c);\beta)$, represented by a normalized constant function on its support ${\rm supp}(\rho^{(c)})\subset h_{0}^{-1}(1)$. We further simplify the notation by using $T_{h}=T_{C(h)}$ and chose some arbitrary base vertex $i(h)$ in the tree $T_{h}$ for each relevant height function $h$. 
\begin{lem}
\label{lemma-limit-robust} The contribution along $I$ to the current $Q_\beta(\gamma)$ in the low-temperature limit is given by
\begin{eqnarray}
\label{limit-arc-contr} \lim_{\beta\to\infty}\int_{I}{J}ds=\sum_{j\in h_{0}^{-1}(1)}{Q}_{i(h)}^{T_{h},j}\left(\rho_{j}^{(b)}-\rho_{j}^{(a)}\right).
\end{eqnarray}
\end{lem}

\begin{proof}
Using the explicit expression for the current based on the Kirchhoff theorem 
given Eq.~\eqref{eq:J-formula}
one can show
\begin{eqnarray}
\label{arc-contr-transformed} \int_{I}{J}ds &=&\sum_{j\in h_{0}^{-1}(1)}{Q}_{i(h)}^{T_{h},j}\rho_{j}^{B}\sum_{T\supset T_{h}}\varrho_{T}^{B}|_{a}^{b} \nonumber \\ &-&\sum_{j\in h_{0}^{-1}(1)}{Q}_{i(h)}^{T_{h},j}\int_{I}ds\rho_{j}^{B}\frac{d}{ds}\sum_{T\supset T_{h}}\varrho_{T}^{B} \nonumber \\ &+&\sum_{(j,T)\in {\cal K}_{h}}{Q}_{i(h)}^{T,j}\rho_{j}^{B}\varrho_{T}^{B}|_{a}^{b} \nonumber \\ &-&\sum_{(j,T)\in {\cal K}_{h}}{Q}_{i(h)}^{T,j}\int_{I}ds\rho_{j}^{B}\dot{\varrho}_{T}^{B} \, ,
\end{eqnarray}
where ${\cal K}_{h}=\left\{(j,T):h_{0}(j)=2 \; {\rm or} \; T_{h}\nsubseteq T\right\}$. To derive Eq.~\eqref{arc-contr-transformed} we first apply integration by parts to the 
explicit expression for the current, followed by representing the sum over the graph vertices $j$ and spanning trees $T$ as a sum over $(j,T)$ with $h_{0}(j)=1$ and $T\supset T_{h}$, and the remaining terms. We also make use of the fact that provided $j\in h_{0}^{-1}(1)$, which implies $j\in \left(T_{h}\right)_{0}$, and $T_{h}\subset T$, we have $Q_{i(h)}^{T,j}=Q_{i(h)}^{T_{h},j}$, i.e., the contribution to the current does not depend on the spanning tree $T$. Then Eq.~\eqref{limit-arc-contr} follows from the following properties that hold inside $Y_{h}$, and are verified directly. For $j\in h_{0}^{-1}(2)$ we have
\begin{eqnarray}
\label{Boltzmann-limit-vertices} \lim_{\beta\to \infty}\rho_{j}^{B}=0, \;\;\; \lim_{\beta\to \infty}d\rho_{j}^{B}=0
\end{eqnarray}
and for $T\nsupseteq T_{h}$ we have
\begin{eqnarray}
\label{Boltzmann-limit-trees} \lim_{\beta\to \infty}\varrho_{T}^{B}=0, \;\;\; \lim_{\beta\to \infty}d\varrho_{T}^{B}=0\, .
\end{eqnarray}
Since $\sum_{T}\varrho_{T}=1$, the properties given by Eq.~\eqref{Boltzmann-limit-trees} also imply
\begin{eqnarray}
\label{Boltzmann-limit-trees-2} \lim_{\beta\to \infty}\sum_{T\supset T_{h}}\varrho_{T}^{B}=1, \;\;\; \lim_{\beta\to \infty}d\sum_{T\supset T_{h}}\varrho_{T}^{B}=0
\end{eqnarray}
Eq.~(\ref{limit-arc-contr}) is obtained by applying the properties given by Eqs.~\eqref{Boltzmann-limit-vertices}, \eqref{Boltzmann-limit-trees}, and \eqref{Boltzmann-limit-trees-2} to the integral expression of Eq.~\eqref{arc-contr-transformed}.
\end{proof}

\begin{defn} \label{defn:check-Q} Let $\check Q$ be the locally constant function given
 by the composite
\[ 
L \check {\cal M}_{\Gamma} @>>>
 H_1(\check {\cal M}_{\Gamma};\Bbb Z) @> {\check q}_* >>
H_1(\Gamma;\Bbb Z) 
 \]
 in which the first map is defined by sending a free loop  to its
integer homology class.
\end{defn}

\begin{proof}[Proof of Theorems \ref{thm:strongPQT} and \ref{thm:realization}]
Let $I_{1},\ldots,I_{k}$ be a simplicial decomposition of $S^{1}$ into closed arcs, with $1,\ldots,k\in\mathbb{Z}/k$, and $I_{m}=[a_{m-1},a_{m}]$, so that $\gamma(I_{m})\subset Y_{h^{m}}$ for some set $h^{1},\ldots h^{k}$ of (extended) height functions. Applying Lemma~\ref{lemma-limit-robust}, and more specifically Eq.~(\ref{limit-arc-contr}), followed by re-grouping the terms in the sum over the arcs we obtain
\begin{eqnarray}
\label{limit-PQT-robust-2} \lim_{\beta\to\infty}{Q}_{\beta}(\gamma)=\sum_{m=1}^{k}\sum_{j}
\left({Q}_{i(h^{m})}^{T_{h^{m}},j}-{Q}_{i(h^{m+1})}^{T_{h^{m+1}},j}\right)\rho_{j}^{(a_{m})}\, .
\end{eqnarray}
The expression in the parenthesis on the right side 
of Eq.~\eqref{limit-PQT-robust-2} does not depend on $j$. Indeed,
this assertion needs only to be checked for for another vertex $j'$ which lies in 
$T_{h_m} \cap T_{h_{m+1}}$. In this instance the unique path running from $j$ to 
$j'$ which is contained in $T_{h_m} \cap T_{h_{m+1}}$
determines a one-chain $c$ such that ${Q}_{i(h^{m})}^{T_{h^{m}},j'} = 
{Q}_{i(h^{m})}^{T_{h^{m}},j} + c$ and likewise ${Q}_{i(h^{m+1})}^{T_{h^{m+1}},j'}
= {Q}_{i(h^{m+1})}^{T_{h^{m+1}},j} +c$. Hence, 
${Q}_{i(h^{m})}^{T_{h^{m}},j}-{Q}_{i(h^{m+1})}^{T_{h^{m+1}},j} = {Q}_{i(h^{m})}^{T_{h^{m}},j'}-{Q}_{i(h^{m+1})}^{T_{h^{m+1}},j'}$.

If we also account for the normalization condition for $\rho^{(a_{m})}$, we can replace summation over $j$ by choosing any vertex $j_{m}\in (h_{0}^{m})^{-1}(1)\cap (h_{0}^{m+1})^{-1}(1)$ and then recast Eq.~ \eqref{limit-PQT-robust-2} in the form
\begin{eqnarray}
\label{limit-PQT-robust-3} \lim_{\beta\to\infty}{Q}_{\beta}(\gamma)\, \, =\,\, \sum_{m=1}^{k}
\left({Q}_{i(h^{m})}^{T_{h^{m}},j_{m}}-{Q}_{i(h^{m+1})}^{T_{h^{m+1}},j_{m}}\right)\, .
\end{eqnarray}
The right side of Eq.~\eqref{limit-PQT-robust-3} is clearly an integer valued one-chain, so
the proof of Theorem \ref{thm:strongPQT} is complete.

We now turn to the proof of Theorem \ref{thm:realization}. With respect to the above
situation,
consider the free loop $\ell:S^{1}\to |\Gamma|$ defined as follows: when the parameter $s\in S^{1}$ changes from the center of the arc $I_{m}$ to its end $a_{m}$, $\ell(s)$ goes from $i(h^{m})$ to $j_{m}$ along the unique minimal length path in the tree $T_{h^{m}}$. When $s$ changes from $a_{m}$ to the center of the arc $I_{m+1}$, $\ell(s)$ goes from $j_{m}$ to $i(h^{m+1})$ along the unique minimal length path in the tree $T_{h^{m+1}}$. It is easy to see that the right side of Eq.~\eqref{limit-PQT-robust-3}, considered as an element of $H_{1}(|\Gamma|)$, is the image of $\ell$ under the map $L|\Gamma|\to H_{1}(|\Gamma|)$ that associates with a free loop its corresponding homology class. On the other hand it is also easy to see that $(\gamma, \ell) \: S^1 \to (\check{{\cal M}}_{\Gamma}\times |\Gamma|)$ has image in $\check{N}\subset \check{{\cal M}}_{\Gamma}\times |\Gamma|$, and we infer $(\gamma, \ell) \in L\check{N}$. By a straightforward inspection of the definitions we see that the right side of Eq.~\eqref{limit-PQT-robust-3} is given by $\check{{Q}}(\gamma)$ (as defined above in 
Definition \ref{defn:check-Q}), which completes the proof. 
\end{proof}

\section{The Chern class description}
\label{sec:current-homological}


\subsection*{The canonical torus}
Set
\[
C^i(\Gamma;U(1)) := U(1)^{\Gamma_i} \qquad i = 0,1\, ,
\]
where the right side denotes the set
of functions $\Gamma_i\to U(1)$. 
The Lie group 
\[ G_{\Gamma} \,\, := \,\, C^0(\Gamma;U(1))
\] 
is called the {\it gauge group}; 
it  acts on $C^1(\Gamma;U(1))$. The action is defined by 
\[
(g\cdot f)(\alpha) = g(d_0(\alpha))f(\alpha)g(d_1(\alpha))^{-1}\, ,
\] where $g\in G_{\Gamma}$ and 
$f\in C_1(\Gamma;U(1))$. Let 
\[
H^1(\Gamma; U(1))
\] 
denote the orbit space of this action
(alternatively, let $\delta\: C_0(\Gamma;U(1)) \to C_1(\Gamma;U(1))$
be given by $\delta(g)(\alpha) = g(d_0(\alpha))g(d_1(\alpha))^{-1}$,
then  $H^1(\Gamma; U(1))$ is the cokernel of $\delta$).
Then $H^1(\Gamma; U(1))$ is an $n$-torus where $n$ is the first Betti number of
$\Gamma$ (this is the torus $S(\Gamma)$ appearing in  Theorem \ref{thm:Chern_description}).

Observe that an element of $H^1(\Gamma; U(1))$ is represented by a function $\lambda\: \Gamma_1 \to U(1)$.
 
\begin{lem} \label{lem:preferred_iso} There is a preferred isomorphism
\[
H^1(H^1(\Gamma;U(1));\Bbb Z) \,\, \cong \,\, H_1(\Gamma;\Bbb Z)\, .
\]
\end{lem}

\begin{proof} For $\alpha \in \Gamma_1$, let $\pi_\alpha\:C^1(\Gamma;U(1)) \to U(1)$
denote the coordinate function given by restriction to $\{\alpha\} \subset \Gamma_1$
(use $C^1(\{\alpha\};U(1)) = U(1)$). The operation $\alpha \mapsto \pi^\alpha$ extends linearly to an isomorphism of 
abelian groups 
\[
C_1(\Gamma;\Bbb Z) \cong [C^1(\Gamma;U(1)),U(1)] = H^1(C^1(\Gamma;U(1));\Bbb Z)\, .
\]
We also have a similar isomorphism $C_0(\Gamma;\Bbb Z) \cong 
H^1(C^0(\Gamma;U(1));\Bbb Z)$.
With these identifications, the boundary operator 
$\partial\: C_1(\Gamma;\Bbb Z) \to C_0(\Gamma;\Bbb Z)$
is given by restriction 
$\delta^*\: H^1(C^1(\Gamma;U(1));\Bbb Z) \to H^1(C^0(\Gamma;U(1));\Bbb Z)$. Hence
$H_1(\Gamma;\Bbb Z)$ is identified with the kernel of $\delta^*$. But the inclusion
$H^1(H^1(\Gamma;U(1));\Bbb Z)\subset \text{\rm ker}(\delta^*)$ is clearly an isomorphism.
\end{proof}


\subsection*{A combinatorially defined line bundle}  
We refer the reader to discussion of \S\ref{sec:weak_map}, especially Remark 
\ref{rem:tree_notation}. Recall that 
\[
\check {\cal M}_{\Gamma} = \bigcup_{h} Y_h
\]
is a covering by open sets where $h = (h_0,h_1)$ ranges over extended height functions.
Associated with $Y_h$ one has a tree $T_h := T_{C(h)}$ such that $h_0^{-1}(1) \subset 
T_{C(h)}$. Fix a basepoint vertex
$i$ for $T_h$ (cf.\ Lemma \ref{lemma-limit-robust})

For $(E,W) \in Y_h$ and $\lambda \in C^1(\Gamma;U(1))$,
we associate a complex line in $\Bbb C^n$, where $n$ is the cardinality of $\Gamma_1$.
For any vertex $j$ of the tree
$T_h$, we have a minimal path $P_i^{T_h,j}$ from $i$ to $j$
which is contained in $T$; this path defines 
the integer value 1-chain $Q_i^{T_h,j}$ (cf.\ Remark \ref{sum-formula-for-A}). 

Let $\Bbb C[\Gamma_0]$ denote the complex vector space with basis $\Gamma_0$. 
Then we obtain a non-zero vector
\begin{equation} \label{eq:path-formula}
v = v(h,E,W,\lambda) := \sum_{j \in (T_h)_0} (e^{-\beta E_j}\prod_{\alpha\in P_{i}^{T_h,j}}
\lambda^{s(\alpha)}_{\alpha})j \,\, \in \,\,  \Bbb C[\Gamma_0]
\end{equation}
where $s(\alpha) = \pm 1$ according as to whether the direction of the path coincides with the orientation of $\alpha$ (this sign coincides with the coefficient appearing of $\alpha$ in $Q_i^{T_h,j}$).

Let $\Bbb Cv \subset \Bbb C[\Gamma_0]$ denote the complex line spanned by this vector.

\begin{lem} \label{lem:Cv} If we choose a different basepoint vertex
the complex line $\Bbb Cv$ remains unchanged.
\end{lem} 

\begin{proof} The amalgamation of the minimal length paths $P_i^{T_h,i'}$ 
and $P_{i'}^{T_h,j}$ produces a new path $P_i^{T_h,i'}P_{i'}^{T_h,j}$ from $i$ to $j$. 
If this path is
minimal, then it is $P_i^{T_h,j}$ and clearly, we have
\[
\prod_{\alpha\in P_{i}^{T_h,j}}\lambda^{s(\alpha)}_{\alpha} = (\prod_{\alpha\in P_{i}^{T_h,i'}} \lambda^{s(\alpha)}_{\alpha}) (\prod_{\alpha\in P_{i'}^{T_h,j}}\lambda^{s(\alpha)}_{\alpha} )\, .
\]
If the amalgamated path isn't minimal, then this formula still holds because the factors
corresponding to indices occurring more than once cancel. This gives independence with respect to the basepoint vertex, as the first factor on the right is independent of $j$.
\end{proof}

For fixed $h$ the assignment $(E,W,\lambda) \mapsto {\Bbb C}v(h,E,W,\lambda)$ describes a line
bundle $\tilde{\xi}_h$ over $C^1(\Gamma;U(1))\times Y_h$. In fact, it is straightforward to check
that $\tilde{\xi}_h$ is trivializable.
We now use the clutching construction to glue  these line bundles together as $h$ varies. This will produce a line bundle over $\tilde{\xi}$ over
$C^1(\Gamma;U(1))  \times \check {\cal M}_{\Gamma}$.
To check this, it suffices to establish the following.

\begin{lem} \label{transition} Given height functions $h$ and $h'$, let 
\[ a\: C^1(\Gamma;U(1)) \times (Y_h \cap Y_{h'}) \to C^1(\Gamma;U(1)) \times Y_h
\]
and  
\[ b\: C^1(\Gamma;U(1)) \times (Y_h \cap Y_{h'}) \to C^1(\Gamma;U(1)) \times Y_{h'}
\] denote the inclusions. 
Then there is an isomorphism of line bundles
$\phi_{ab}\: b^* \tilde{\xi}_{h'} @> \cong >> a^*{\tilde \xi}_h $.
Furthermore, this isomorphism satisfies the cocycle condition
$\phi_{ac} = \phi_{ab}\phi_{bc}$.
\end{lem}
\begin{proof}  Associated with $(\lambda,E,W) \in C^1(\Gamma;U(1)) \times Y_h$ 
and a basepoint vertex $i$ for $T_h \cap T_{h'}$, we have
a non-zero vector $v$ which is defined by  Eq.~\eqref{eq:path-formula}.  To indicate
the dependence of this vector on $h$, let us redenote it  by $v_h$. Similarly,
for $(\lambda,E,W) \in C^1(\Gamma;U(1)) \times Y_{h'}$ we have $v_{h'}$.
Then  define $\phi_{ab}(z\cdot v_{h'}) := z\cdot v_{h}$.
The cocycle condition is then immediate.
\end{proof}

Let the gauge group $G_{\Gamma}$ act diagonally
$C^1(\Gamma;U(1))  \times \check {\cal M}_{\Gamma}$ (where the action on the second factor
is trivial). Then $G_{\Gamma}$ also acts in an evident way on the total space $E(\tilde{\xi})$ of the line bundle $\tilde{\xi}$ equipping it with the structure of a  $G_\Gamma$-equivariant line bundle.
 Taking orbit spaces defines a line bundle
$\xi$ over $H^1(\Gamma;U(1)) \times \check{{\cal M}}_{\Gamma}$. 
If $\pi\: C^1(\Gamma;U(1)) \to H^1(\Gamma;U(1))$ is the quotient map, then $\tilde{\xi}$ is given by
the base change of $\xi$ along
\[\pi\times \text{id}\: C^1(\Gamma;U(1)) \times \check{{\cal M}}_{\Gamma} 
\to H^1(\Gamma;(U(1))\times \check{{\cal M}}_{\Gamma}\, .
\]
Naturality of Chern classes gives a commutative diagram
\begin{equation}\label{eq:transfer_diagram}
\xymatrix{
H_1(\check {\cal M}_{\Gamma};\Bbb Z) \ar[rr]^{c_1(\tilde{\xi})/} \ar[rrd]_{c_1(\xi)/}
 && C_1(\Gamma;\Bbb Z)  \\
&&  H_1(\Gamma;\Bbb Z) \, .\ar@{^{(}->}[u]_{\pi^*}
}
\end{equation}
Here we have used the preferred isomorphism  $H^1(H^1(\Gamma;U(1));\Bbb Z) = H_1(\Gamma;\Bbb Z)$
of Lemma \ref{lem:preferred_iso} as well as
a similarly constructed identification  $H^1(C^1(\Gamma;U(1));\Bbb Z) = C_1(\Gamma;\Bbb Z)$.
With respect to these identifications, $\pi^*$, which is the 
homomorphism induced by $\pi$ on first integer cohomology, is just
the canonical
inclusion $H_1(\Gamma;\Bbb Z) \subset C_1(\Gamma;\Bbb Z)$.

\begin{thm} \label{thm:combinatorial_chern} The homomorphism $c_1(\xi)/$ coincides with ${\check q}_*$.
\end{thm}

In order to prove Theorem \ref{thm:combinatorial_chern} we digress to
explain how holonomy relates to the homomorphism given by
slant product with the first Chern class. If $\xi$ is a complex line bundle
over a connected space $B$ and structure group $U(1)$,
we have the holonomy map
\begin{equation}
h_\xi\: LB \to U(1)\, .
\end{equation}
In fact, $h_\xi$ can be chosen in such a way that
if we choose a basepoint of $B$ and restrict
$h_\xi$ to the based loop space $\Omega B$, 
we can deloop to a map 
$B \to BU(1) = {\Bbb C}P^\infty$ that classifies the bundle $\xi$ 
and hence the Chern class $c_1(\xi)$. Consequently, if we choose
$h_\xi$ in this way, it determines the Chern class.

Now suppose $B = X \times Y$. Then we can restrict $h_\xi$ to the subspace
$X \times LY \subset LX \times LY = L(X\times Y)$ and take the adjoint
to  obtain a map
\[
LY \to F(X,U(1))\, ,
\]
where $F(X,U(1))$ is the function space of maps from $X$ to $U(1)$.
Then the diagram
\begin{equation}
\xymatrix{
LY \ar[r]\ar[d] &  F(X,U(1))\ar[d] \\
H_1(Y;\Bbb Z) \ar[r]_{c_1(\xi)/} & H^1(X;\Bbb Z)
}
\end{equation}
commutes, where the left vertical map sends a loop to its homology class and the
right vertical map sends a function to its homotopy class considered as
an element of $H^1(X;\Bbb Z) = [X,U(1)]$.

\begin{proof}[Proof of Theorem \ref{thm:combinatorial_chern}]
 By Eq.~\eqref{eq:transfer_diagram} it suffices to show that 
$c_1(\tilde{\xi})/$ coincides with the homomorphism 
\[
H_1(\check {\cal M}_{\Gamma};\Bbb Z) @> {\check q}_* >>  H_1(\Gamma;\Bbb Z)
@> \pi^* >> C_1(\Gamma;\Bbb Z)  \, .
\]
Suppose we are given $(\lambda,\gamma)\in C^1(\Gamma;U(1)) \times L\check {\cal M}_{\Gamma}$.
As done previously, we partition $S^1$ into closed 
arcs $[a_{k},a_{k+1}]$ for $0\le k\le n$ with $a_{n+1} \equiv a_0$,
such that the projection of such an arc into  $\check {\cal M}_{\Gamma}$ is contained in
a neighborhood of type $Y_{h_{k-1}}$.  Choose a basepoint vertex $i_k$ lying in the 
intersection $T_{k} \cap T_{k+1}$, where $T_k$ denotes
the tree associated with $Y_{h_{k}}$ . Then we have a minimal length path 
$P_{i_k}^{T_k,i_{k+1}}$ from $i_k$ to $i_{k+1}$ and the product
\begin{equation}\label{eq:holonomy}
\prod_{k=0}^{n} \prod_{\alpha \in P_{i_k}^{T_k,i_{k+1}}} \lambda^{s(\alpha)}_\alpha  \in U(1)
\end{equation}
describes a map $C^1(\Gamma;U(1)) \to U(1)$ that gives the holonomy 
around $\gamma$ (here we are using Lemmas \ref{lem:Cv} and \ref{transition}).

In the special instance of $\lambda\in C^1(\Gamma; U(1))$ which is identically $1$ except for a single edge $\alpha$, the value of the map $C^1(\Gamma;U(1)) \to U(1)$
at $\lambda$ is given by $\lambda_\alpha^q$ where $q$ represents the net number of times $\alpha$ is traversed, with orientation taken into account. Consequently, if we identify 
$H^1(C^1(\Gamma;U(1));\Bbb Z)$ with $C_1(\Gamma;\Bbb Z)$, then $\lambda_\alpha^q$
is identified with the chain $q\alpha$. It follows that
the map $C^1(\Gamma; U(1)) \to U(1)$ defined by Eq.~\eqref{eq:holonomy} corresponds
to the integer cycle in $C_1(\Gamma;\Bbb Z)$ given by 
\begin{equation} \label{eq:holonomy-additive}
\sum_{k=0}^{n} \sum_{\alpha \in P_{i_k}^{T_k,i_{k+1}}} s(\alpha)\alpha \,\, := \,\,  
\sum_{k=0}^{n} Q_{i_k}^{T_k,i_{k+1}}\, .
\end{equation}
On the other hand, the paths $P_{i_k}^{T_k,i_{k+1}}$ describe a lift
of $\gamma\: S^1 \to \check {\cal M}_{\Gamma}$ through the space 
$\check N$ appearing Eq.\ \eqref{define-check-N} (roughly,  one defines the lift
by mapping the midpoint $a'_k$ of the arc $[a_{k},a_{k+1}]$ to the point 
$(\gamma(a'_k),i_k)$ and 
uses  $P_{i_k}^{T_k,i_{k+1}}$ to connect these points). Then application of
the projection map $p_2 \: \check N \to |\Gamma|$ to the given lift 
produces a map $S^1 \to |\Gamma|$ that represents $\check q_*([\gamma]) \in H_1(\Gamma;\Bbb Z) 
\subset C_1(\Gamma;\Bbb Z)$.
From this description it is straightforward to check that $\check q_*([\gamma])$
coincides with the element defined by Eq.~\eqref{eq:holonomy-additive} (cf.\ Eq.~\eqref{eq:expression_for_current}). 
\end{proof}

\section{The ground state bundle: a conjecture}

By coupling the master operator with elements of the torus $H^1(\Gamma;U(1))$ one
can  extend the master operator to a self-adjoint operator over the complex numbers. 
This extension is called the twisted master operator;
its eigenvalues are real and non-positive. The eigenspace associated with the maximum
non-zero eigenvalue is called the {\it ground state.}
One may use the twisted master operator to  
 define another weak complex line bundle $\eta$, this time over
$H^1(\Gamma;U(1))  \times \tilde {\cal M}_{\Gamma}$, where $\tilde {\cal M}_{\Gamma} \subset
 {\cal M}_{\Gamma}$ is characterized by the condition that the ground state at
 each point is non-degenerate, meaning that it has rank one.
Then roughly, $\eta$ is defined
by taking the ground state at each point of the base.
We call this the {\it ground state bundle}.
Arguments from physics suggest that the ground state bundle is
equivalent to the weak complex
line bundle $\xi$ that was defined in the previous section.  In what follows
we will formulate this idea as a pair of conjectures.

\subsection*{The twisted master operator}
The {\it twisted master operator,} defined below, is a smooth map
 \[
 \bar H\: C^1(\Gamma;U(1)) \times \cal M_{\Gamma} \to  \End_{\Bbb C}(C_{0}(X;\mathbb{C}))\, ,
 \] 
where $C_{0}(X;\mathbb{C})$ is the complex vector space with basis
$\Gamma_0$.  It extends the master operator in the sense that
 \[
 \bar H(1,\beta,E,W) = H(\beta,E,W)
 \]
 where $1 \in C^1(\Gamma;U(1))$ is the function with constant value $1\in U(1)$ and where
 we are interpreting the right side of this identity using extension of scalars.

For   $\lambda \in C^1(\Gamma;U(1))$, let $\hat \lambda\: C_1(\Gamma;\Bbb C) \to 
C_1(\Gamma;\Bbb C)$ be given by rescaling each basis element $\alpha$ by $\lambda(\alpha)\alpha$.
Then
\[
\bar H(\lambda,\beta,E,W) := -\partial \hat g^{-1} \hat \lambda \partial^* \hat\kappa\, .
\]
It is clear from the definition that $\bar H$ is self-adjoint. In particular, its eigenvalues
are all real.

Let the gauge group $G_{\Gamma}$ act on $\End_{\Bbb C}(C_{0}(X;\mathbb{C}))$ 
 via conjugation and trivially on $\cal M_{\Gamma}$. The following is then
 a formal consequence of the definitions.

\begin{lem}[Gauge Symmetry]  The twisted master operator
is $G_{\Gamma}$-equivariant, i.e., for $h\in G_{\Gamma}$, we have
\[
\bar H(h\cdot\lambda,\beta,E,W) = h \cdot \bar H(\lambda,\beta,E,W)\, .
\]
In particular, 
for each $(E,W) \in \cal M_{\Gamma}$, 
the spectrum of  $H(\beta,E,W) = \bar H(1,\beta,E,W)$
is invariant with respect to the action of the gauge group.
\end{lem}

\begin{defn} Let $A\: V \to V$ be a self-adjoint linear transformation of a finite dimensional
complex vector space $V$, all of whose eigenvalues are non-positive. Then the {\it ground state} $L$ is the eigenspace for maximal
eigenvalue of $A$.  We say that $A$ has {\it nondegenerate} its ground state $L$ has rank one.
\end{defn}

\subsection*{An analytically defined weak line bundle}
Define an open subset 
\[ \tilde {\cal M}_{\Gamma} \subset \Bbb R_+ \times \cal M_\Gamma
\] to be
the set of those $(\beta_0,E,W)$ such that for every $\lambda \in C^1(\Gamma;U(1))$
and every $\beta \ge \beta_0$,
the twisted master operator
\[
\bar H(\lambda,\beta,E,W)\: C_0(\Gamma;\Bbb C) \to C_0(\Gamma;\Bbb C)
\] has a non-degenerate
ground state. 

For $(\lambda,\beta_0,E,W) \in C^1(\Gamma;U(1)) \times \tilde {\cal M}_{\Gamma}$,
let us denote the ground state of the twisted master operator by 
$\Bbb L(\beta_0,\lambda,E,W)$; it is a complex line in $C_0(\Gamma;\Bbb C)$. Consider
$$
E = \{(\lambda,\beta_0, E,W,v)| v \in \Bbb L(\beta_0,\lambda,E,W)\}
$$
which is topologized as a subspace of 
$C^1(\Gamma;U(1)) \times \tilde {\cal M}_{\Gamma} \times 
C_0(\Gamma;\Bbb C)$. Then we have an evident projection 
$$
p\: E  \to C^1(\Gamma;U(1))  \times \tilde {\cal M}_{\Gamma}  \, .
$$

\begin{lem} \label{lem:line-bundle} The map $p$ is a smooth complex line bundle projection.
\end{lem}

\begin{proof} Let $V = C_0(\Gamma;\Bbb C)$ and let
 $L^1(V,V)$ denote 
the space consisting of complex linear self-maps of $V$  having
corank one. Then $L^1(V,V)$ is a smooth manifold of real dimension
$2|\Gamma_0|^2 - 2$ (see \cite[prop.\ 5.3]{GG}). 
 The operation which sends a complex linear self-map
 to its null space defines a smooth map 
\[
L^1(V,V) @> \text{ker} >> \Bbb P^1(V)
\]
whose target is the projective space of complex lines in $V$.
The composition 
\[
C^1(\Gamma;\Bbb C) \times \tilde {\cal M}_{\Gamma} @>\bar H >> L^1(V,V) @> \text{ker} >> \Bbb P^1(V)
\]
is therefore smooth and the pullback of the tautological line bundle
over $\Bbb P^1(V)$ gives the projection $p$.
\end{proof}

Let $\pi\: \tilde {\cal M}_{\Gamma} \to {\cal M}_{\Gamma}$ be
given by the projection $(\beta_0, E,W) \mapsto (E,W)$.

\begin{conj} \label{lem:projection}  The image of $\pi$ is $\check {\cal M}_{\Gamma}$, and
$\pi\: \tilde {\cal M}_{\Gamma} \to \check {\cal M}_{\Gamma}$ is a weak
homotopy equivalence.
\end{conj}

Let $\tilde{\eta}$ be the complex line bundle defined by Lemma \ref{lem:line-bundle}.
The gauge group $G_\Gamma$ acts on both the total and base spaces making
$\tilde{\eta}$ into a $G_\Gamma$-equivariant complex line bundle. Taking orbits, we
obtain a complex line bundle $\eta$ over 
$H^1(\Gamma;U(1))  \times \tilde {\cal M}_{\Gamma}$.

Let $h\: H^1(\Gamma;U(1))  \times \tilde {\cal M}_{\Gamma} \to H^1(\Gamma;U(1))  \times 
\check {\cal M}_{\Gamma}$ be given by $\text{id} \times \pi$.
Then using
Lemma \ref{lem:projection}, the pair $(\eta,h)$ is a weak complex line bundle over
$H^1(\Gamma;U(1))  \times 
\check {\cal M}_{\Gamma}$.





Then slant product with the first Chern class of $\eta$ gives a homomorphism
\[
c_1(\eta)/ \: H_1(\check{\cal M}_{\Gamma};\Bbb Z) \to H_1(\Gamma;\Bbb Z)\, ,
\]
where we have implicitly used the identification $H_1(\check{\cal M}_{\Gamma};\Bbb Z) \cong
H_1(\tilde {\cal M}_{\Gamma};\Bbb Z)$ of Lemma \ref{lem:projection}  and also the identification
 $H^1(H^1(\Gamma;U(1));\Bbb Z)\cong H_1(\Gamma;\Bbb Z)$ of
Lemma \ref{lem:preferred_iso}.

\begin{conj} \label{conj:eta=xi} The homomorphisms $c_1(\eta)/$ and $c_1(\xi)/$ coincide.
\end{conj}

\section{Appendix: an adiabatic theorem \label{adiabatic_section}}
Here we formulate and prove an adiabatic theorem for periodic driving. Roughly,
it states that for slow enough driving a periodic
solution of the master equation exists and is unique, and furthermore, in the adiabatic 
limit this solution will tend to the Boltzmann distribution taken at the associated normalized
driving protocol.

Let us introduce the {\it evolution operator} $U(t,t_0) = U(t,t_0;H,\tau_D)$ for 
$0 \le t_0 \le t \le 1$,
which is the unique solution to the initial value problem
\[
\frac{d}{dt}U(t,t_0) = \tau_D U(t,t_0)H(\gamma(t)), \qquad U(t_0,t_0) = I\, ,
\]
where $I$ denotes the identity operator. 
We remark that $U(t,t_0)$ is also called the {\it path-ordered exponential}
and is sometimes expressed in the notation 
\[
\hat T\exp(\tau_D\int_{t_0}^{t} dt' H(\gamma(t')))
\]
(cf.\ \cite{Lam}).

Then it is elementary to show that the master equation
\[
\dot {\mathbf p}(t) = \tau_D H(\gamma(t)){\mathbf p}(t)
\]
has formal solution
\[
\mathbf p(t)Ê= U(t,0) \mathbf p(0)\, .
\]

\begin{prop}\label{periodic_solution} Let 
$(\tau_D,\gamma)$ be a periodic driving protocol. Then there
is positive real number $\tau_0$ such that if $\tau_D \ge \tau_0$, then
there is a unique periodic solution $\rho(t)$ to the master equation, i.e.,
$\rho(0) = \rho(1)$.
\end{prop} 

\begin{proof} We shall use abbreviated notation and write 
$\rho^{\text{\rm B}}(t)$ in place of
$\rho^{\text{\rm B}}(\gamma(t))$.
For any solution to the master equation ${\mathbf p}(t)$,
set
\[
\xi(t) := \mathbf  \rho^{\text{\rm B}}(t) - {\mathbf p}(t)\, .
\]
Then $\xi\:[0,1]\to \tilde C_0(\Gamma;\Bbb R)$ is a family of
reduced population vectors.
Furthermore, $\xi(t)$ is periodic if and only ${\mathbf p}(t)$ is, and
\begin{equation} \label{xi-defn}
{\mathbf p}(t) = \mathbf  \rho^{\text{\rm B}}(t) + \xi(t)\, .
\end{equation}

Inserting Eq.~\eqref{xi-defn}
expression into the master equation and using the fact that the 
Boltzmann distribution lies in the null space of the master operator, 
we obtain the first order linear differential equation in $\xi$,
\begin{equation} \label{diffEQ}
\dot \xi(t) - \tau_D H(t) \xi(t) = -\dot \rho^{\text{\rm B}}(t) \, ,
\end{equation}
where $H(t)$ is shorthand for $H(\gamma(t))$.

Applying $U(1,t)$ to Eq.~\eqref{diffEQ} we get
\[
U(1,t)\dot \xi - U(1,t) \tau_D H \xi = - U(1,t)
\dot \rho^{\text{\rm B}} \, .
\]
Notice that the left side of the last display is just $d/dt ( U(1,t)\xi)$.
Integrating both sides  we obtain
\[
 U(1,t) \xi(t) =  -\int_0^t U(1,t')\dot \rho^{\text{\rm B}} \, dt'  
 + C\, .
\]
Setting $t = 0$ we see that $U(1,0)\xi(0) = C$. Evaluating 
at $t=1$ and using the fact that $U(1,1) = I$ yields
\[
\xi(1) -  U(1,0)\xi(0)  = -\int_0^1  U(1,t')\dot \rho^{\text{\rm B}}\,  dt'\, .
\]
Consequently, $\xi(0) = \xi(1)$ if and only if 

\begin{equation}\label{period_condition}
(I-U(1,0))\xi(0) = -\int_0^1 U(1,t')\dot \rho^{\text{\rm B}} \, dt'\, .
\end{equation}
It is therefore sufficient to show that the operator $I-U(1,0)$
is invertible, when considered as an operator acting on 
the invariant subspace $\tilde C_0(\Gamma;\Bbb R)$, provided
 $\tau_D$ is sufficiently large.

Let $\lambda$ and $c$ be the constants obtained in Lemma \ref{estimates}
below. If we set $\tau_0 := (1/\lambda) \ln(2c)$, then we have $\|U(1,0)\| <1/2$.
It follows that $I-U(1,0)$ is invertible on $\tilde C_0(\Gamma;\Bbb R)$.
\end{proof}

The last part of the proof of Proposition \ref{periodic_solution} rested on an
estimate that appears below. To formulate it we use the norm on $\tilde C_0(\Gamma;\Bbb R)$
given by $\|\xi\| = \sqrt{\langle \xi,\xi\rangle}$ where the inner product is
the one induced by the standard inner product on $C_0(\Gamma;\Bbb R)$. 
If $A$ is an operator on $\tilde C_0(\Gamma;\Bbb R)$ then we define
$\|A\| :=\sup_{\xi \ne 0 } \| A\xi \|\|\xi\|^{-1} = 
\sup_{\xi = 1 } \| A\xi \|.$

\begin{lem}\label{estimates}
For a periodic driving protocol $(\tau_D,\gamma)$, there are
positive constants $\lambda$ and $c$ such that for all $t,t_0\in [0,1]$ we
have
\[
\| U(t,t_0) \| < ce^{-\lambda\tau_D(t-t_0)}\, .
\]
\end{lem}

\begin{proof}
Consider the time-dependent inner product 
$\kappa_{t}=\kappa(\gamma(t))$ in $\tilde{C}_{0}(\Gamma;\mathbb{R})$, 
defined by 
$\kappa_{t}(\xi,\eta) = 
\sum_{j\in \Gamma_{0}}e^{\beta E_{j}(t)}\xi_{j}\eta_{j}$ (the is just $\langle {\xi},{\eta}\rangle_{\hat \kappa_t}$ in the notation of Remark \ref{rem:master_operator_remarks}).
Then for all $t$ the operator 
$H(t) = H(\gamma(t))$, when considered as acting on $\tilde C_0(\Gamma;\Bbb R)$, 
is self-adjoint 
with respect to the inner product $\kappa_{t}$ and its spectrum is strictly negative.

Set $\lambda := -\sup_{t\in [0,1]}\sigma(H(t))$, where $\sigma(T)$ denotes the spectrum of a linear operator $T$. Then $\lambda> 0$, and the spectrum of the operator ${H}_{0}(t)={H}(t)+\lambda I$ is non-negative for all $t$. Let $U_{0}(t,t_{0})$ be the corresponding evolution operator. Then $U(t,t_{0})= e^{-\lambda\tau_{D}(t-t_{0})}U_{0}(t,t_{0})$. Hence, $\|{U}(t,t_{0})\|= e^{-\lambda\tau_{D}(t-t_{0})}\|{U}_{0}(t,t_{0})\|$. 
So all we need to prove is that $\|{U}_{0}(t,t_{0})\|$ is uniformly bounded.

Let $\eta(t)$ be the solution of the master equation $\dot{\eta}(t)= 
\tau_{D}H_{0}(t)\eta(t)$ with the initial condition $\eta(t_{0})=\xi$. We then have
\begin{eqnarray}
\label{ME-for-scalar-prod} \frac{d}{dt}\kappa_{t}(\eta(t),\eta(t))&=& \dot{\kappa}_{t}(\eta(t),\eta(t))+ 2\tau_{D}\kappa_{t}({H}_{0}(t)\eta(t),\eta(t)) \nonumber \\  &\le& \dot{\kappa}_{t}(\eta(t),\eta(t)),
\end{eqnarray}
since for all $t$ we have $\kappa_{t}({H}_{0}(t)\eta(t),\eta(t))\le 0$. 

Since $\eta(t)\ne 0$ provided $\eta(t_{0})= \xi\ne 0$,  we infer that
 $\kappa_{t}(\eta(t),\eta(t))> 0$ for all $t$. By compactness, $\|\kappa_{t}\|$ is bounded  below, and since $\|\dot{\kappa}_{t}\|$ is bounded above, we infer that there is a constant $A>0$, 
 so that $\dot{\kappa}_{t}(\eta(t),\eta(t))(\kappa_{t}(\eta(t),\eta(t)))^{-1}< A$. Combined with Eq.~(\ref{ME-for-scalar-prod}) this implies $(d/dt)\ln \kappa_{t}(\eta(t),\eta(t))< A$, and further implies the uniform bound
\begin{eqnarray}
\label{scalar-prod-bound} \frac{\kappa_{t}({U}_{0}(t,t_{0})\xi,{U}_{0}(t,t_{0})\xi)} {\kappa_{t_{0}}(\xi,\xi)}= \frac{\kappa_{t}(\eta(t),\eta(t))}{\kappa_{t_{0}}(\eta(t_{0}),\eta(t_{0}))}< e^{A(t-t_{0})}.
\end{eqnarray}
The uniform bound provided by Eq.~(\ref{scalar-prod-bound}) implies the uniform bound
\begin{eqnarray}
\label{scalar-standard-prod-bound} \frac{\langle {U}_{0}(t,t_{0})\xi,{U}_{0}(t,t_{0})\xi\rangle} {\langle\xi,\xi\rangle}< B^{2},
\end{eqnarray}
with respect to the standard inner product for some $B>0$, which immediately implies the uniform bound $\|{U}_{0}(t,t_{0})\|< B$. \qed
\end{proof}

\begin{cor}[Adiabatic Theorem, cf.\ {\cite[V.3]{vanKampen}}]
\label{cor:adiabatic_theorem} Let $(\tau_D,\gamma)$ be a periodic driving protocol, with $\tau_D$ sufficiently large. 
If $\mathbf \rho(t)$ denotes
the periodic solution
of the master equation, then 
\[
\mathbf  \rho^{\text{\rm B}}(\gamma(t)) = \lim_{\tau_D\to \infty} 
\mathbf \rho(t)\, .
\]
\end{cor}

\begin{proof} It is enough to show
that  $\lim_{\tau_D\to \infty} \xi(t) =0$
where $\xi(t)$ is as in the proof of Proposition \ref{periodic_solution}.
We first show that $\lim_{\tau_D\to \infty} \xi(0) =0$.

To see this, start with the estimate
\begin{align*}
\| \int_0^t U(t,t')\dot \rho^{\text{\rm B}}(t') \, dt' \|
& \le 
 \int_0^t \| U(t,t')\| \| \dot \rho^{\text{\rm B}}(t')\| \, dt' \\
 & \le c\sup_{t' \in [0,1]} \| \dot \rho^{\text{\rm B}}(t')\|\int_0^1 
 e^{-\lambda\tau_D(1-t)}\, dt \\
 & <  \frac{\alpha c}{\lambda \tau_D}\, ,
 \end{align*}
 where $\alpha = \sup_{t' \in [0,1]} \|\dot \rho^{\text{\rm B}}(t')\|$.
 Recalling that $\|U(1,0)\| < 1/2$, we have $\|(I - U(1,0))^{-1} \| < 2$. 
 Consequently, 
 \begin{align*}
 \| \xi(0) \|  & =  \|(I-U(1,0))^{-1}\int_0^1 U(1,t')\dot \rho^{\text{\rm B}} \, dt' \| \\
               & \le \|(I-U(1,0))^{-1}\| \|\int_0^1 U(1,t')\dot \rho^{\text{\rm B}} \, dt' \| 
               \\
               & <  \frac{2\alpha c}{\lambda \tau_D}
   \end{align*}
Therefore, $\lim_{\tau_D\to \infty}\xi(0) = 0$.

The proof that $\lim_{\tau_D\to \infty} \xi(t_0) =0$ for any $t_0 \in [0,1]$ is similar, 
using a suitable modification of the above estimate with $t_0$ in place of $0$ and
Lemma \ref{estimates}) to give a similar bound for  $\| \xi(t_0) \|$ (we omit the details).
\end{proof}

\end{document}